\documentclass[11pt,reqno,a4paper]{amsart}
\pdfoutput 1

\usepackage{amsmath,amssymb,amsthm,amsfonts}
\usepackage{mathtools}
\usepackage{enumitem}
\usepackage{amsfonts}
\usepackage{xspace}
\usepackage{fbox}
\usepackage{enumitem}

    \usepackage[utf8]{inputenc}
    \usepackage{wasysym}
    \usepackage{indentfirst}
    \usepackage{graphicx}
    \usepackage{xfrac}
    \usepackage{esint}
    \usepackage{placeins}   
  
\usepackage[margin=3cm,top=4cm,bottom=4cm,footskip=1cm,headsep=2cm]{geometry}
\usepackage[utf8]{inputenc}
\usepackage{float}
\usepackage[british]{babel}

\definecolor{darkcerulean}{rgb}{0.03, 0.27, 0.49}
\usepackage[colorlinks=true]{hyperref}
\hypersetup{citecolor=darkcerulean,linkcolor=darkcerulean,    urlcolor=darkcerulean}

\usepackage{color}

\usepackage[backend=biber,maxbibnames=5,maxalphanames=5,style=alphabetic,bibencoding=utf8,giveninits,url=false,isbn=false,doi=true,eprint=false]{biblatex}
\renewbibmacro{in:}{}
\AtBeginBibliography{\small}

\def\softd{{\leavevmode\setbox1=\hbox{d}%
		\hbox to 1.05\wd1{d\kern-0.4ex{\char039}\hss}}}%cstocs

\newcommand{\jump}[1]{\ensuremath{\left[\![ #1\right]\!]}}

 \providecommand{\esssup}{\mathop{\mathrm{ess\,sup}}}

\def\div{\operatorname{div}}

\def\la{\langle}
\def\ra{\rangle}

\DeclarePairedDelimiter{\norm}{\|}{\|}

\def\bu{\mathbf{u}}
\def\u{\mathbf{u}}
\def\e{\mathbf{e}}
\def\bz{\mathbf{z}}
\def\bv{\mathbf{v}}
\def\bw{\mathbf{w}}
\def\w{\mathbf{w}}
\def\bf{\mathbf{f}}
\def\bg{\mathbf{g}}

\def\r{\mathbf{r}}

\def\bF{\mathbf{F}}
\def\bm{\mathbf{m}}

\def\cH1Poinc{c_p}
\def\cembA{c_{e1}}
\def\cembB{c_{e2}}
\def\cPiLstabA{c_{\Pi1}}
\def\cPiLstabC{c_{\Pi2}}

\providecommand{\normtmp}[2]{{#1\lVert #2 #1\rVert}}
\providecommand{\norm}[1]{\normtmp{}{#1}}
\providecommand{\bignorm}[1]{\normtmp{\big}{#1}}
\providecommand{\Bignorm}[1]{\normtmp{\Big}{#1}}

\providecommand{\skptmp}[3]{{\ensuremath{#1\langle {#2}, {#3} #1\rangle}}}
\providecommand{\skp}[2]{\skptmp{}{#1}{#2}}

\providecommand{\abstmp}[2]{{#1\lvert #2 #1\rvert}}
\providecommand{\abs}[1]{\abstmp{}{#1}}

\providecommand{\setR}{\mathbb{R}}

\providecommand{\dx}{\,\mathrm{d}x}
\providecommand{\ds}{\,\mathrm{d}s}

\def\dom{\mathbb{T}^3}

\newtheorem{lemma}{Lemma}

\newtheorem{proposition}[lemma]{Proposition}
\newtheorem{theorem}[lemma]{Theorem}
\newtheorem{corollary}[lemma]{Corollary}
\theoremstyle{definition}
\newtheorem{definition}[lemma]{Definition}
\newtheorem{assumption}[lemma]{Assumption}

\newtheorem{remark}[lemma]{Remark}

\def\dt{\partial_t}

\def\dtau{d_\tau}

\def\XXint#1#2#3{{\setbox0=\hbox{$#1{#2#3}{\int}$}\vcenter{\hbox{$#2#3$}}\kern-.5\wd0}}

\makeatletter
\@namedef{subjclassname@2020}{%
	\textup{2020} Mathematics Subject Classification}
\makeatother

\title[A posteriori existence of strong solutions]{A posteriori existence of strong solutions\\ to the Navier--Stokes equations in 3D}

\date{\today}

\author[A. Brunk]{Aaron Brunk}
\author[J. Giesselmann]{Jan Giesselmann}
\author[T. Tscherpel]{Tabea Tscherpel}
\address[J.~Giesselmann, T. Tscherpel]{Department of Mathematics, Technische Universität Darmstadt, Dolivostr.~15, 64293 Darmstadt,	Germany}
\address[A.~Brunk]{Institute of Mathematics, Johannes Gutenberg University, Staudingerweg 9, 55099 Mainz, Germany}
\email{abrunk@uni-mainz.de}
\email{giesselmann@mathematik.tu-darmstadt.de}
\email{tscherpel@mathematik.tu-darmstadt.de}

\begin{document}

%	\begingroup
%	\def\uppercasenonmath#1{} % this disables uppercasing title
%	\let\MakeUppercase\relax % this disables uppercasing authors
%	\maketitle
%	\endgroup
%%\nocite{*}

\begin{abstract}
Global existence of strong solutions to the three-dimensional incompressible Navier--Stokes equations remains an open problem. 
A posteriori existence results offer a way to rigorously verify the existence of strong solutions by ruling out blow-up on a certain time interval, using only numerical solutions. 
In this work we present such a result for the Navier--Stokes equations subject to periodic boundary conditions, which makes use of a version of the celebrated blow-up criterion in the critical space  $L^\infty(L^3)$ by Iskauriaza, Serëgin and Shverak (2003).  
Our approach is based on a conditional stability estimate in $L^2$ and $L^3$. 
The a posteriori criterion that, if satisfied, verifies existence of strong solutions, involves only negative Sobolev norms of the residual. 
We apply the criterion to numerical approximations computed with mixed finite elements  and an implicit Euler time discretisation. 
A posteriori error estimates allow us to derive a fully computable criterion without imposing any extra assumptions on the solution. 
While limited to short time intervals, with sufficient computational resources in principle the criterion might allow for a verification over longer time intervals than what can be achieved by theoretical means. 

\end{abstract}



\maketitle

\setcounter{tocdepth}{1}
\tableofcontents

\section{Introduction}

For the incompressible Navier--Stokes equations in three space dimensions existence of global strong solutions and the possibility of finite-time blow-up of solutions remains an open problem since the pioneering work of Leray~\cite{Leray1934} in 1934. 
Indeed, those questions are at the core of one of the still unresolved Millennium prize problems~\cite{Fefferman2006}, see also~\cite{Farwig2021}. 

While Leray--Hopf weak solutions exist globally in time, existence of strong solutions is %available
guaranteed only for short time intervals (unless the data are assumed to be small). 
Strong solutions are unique in the class of Leray--Hopf weak solutions~\cite{Leray1934, Serrin1963}. 
The first results on the local existence of strong solutions as well as possible finite-time blow-up are due to Leray himself. 
He provides lower bounds for the existence time of a strong solution and several blow-up criteria involving subcritical $L^p$-norms in \cite{Leray1934}. 
Strong solutions and hence absence of blow-up is guaranteed if the so-called Serrin condition holds which involves an additional $L^p(L^q)$ estimate, see~\cite{Serrin1962,Serrin1963}.
An extension to the critical norm $L^\infty(L^3)$ is proven in the seminal work~\cite{IskauriazaSereginShverak2003}. 
In the sequel, many generalisations to different spaces and related blow-up criteria have been derived,  see~\cite{Farwig2014,Gallagher2016,Maremonti2018,Chemin2019}; 
for a localised version see also~\cite{AlbrittonBarker2020};
we refer to~\cite{BarkerPrange2024} for a recent survey, and to~\cite{Farwig2021} for a broader review on the incompressible Navier--Stokes equations. 
Those criteria constitute conditional existence and regularity results, and can be seen as realisations of the following suggestion by Nash~\cite[p.~993]{Nash1958}
\begin{quote}
	\textit{Probably one should first try to prove a conditional existence and uniqueness theorem for the flow equations. This should give existence, smoothness, and unique continuation (in time) of flows, conditional on the non-appearance of certain gross types of singularity, [...]. 
	} 
\end{quote}

One way to approach the question of the existence of strong solutions is through the use of large-scale numerical simulations. 
Experimental investigation of the many blow-up criteria and of the scaling behaviour near blow-up has been performed, among others in~\cite{Grundy1999,Hou2009, Hou2023}. 
While such experiments may give some intuition, they do not permit rigorous statements about the existence time of strong solutions. 
In particular, one usually cannot distinguish between the effects of low regularity of the solution and the discretisation error. 
\smallskip 

\textbf{A posteriori existence.} Alternatively, one may rigorously show existence of strong solutions to the 3D incompressible Navier--Stokes equations by use of an \emph{a posteriori criterion}.
This means that it can be tested without any knowledge of the exact solution to the Navier--Stokes equations since it involves only computable quantities depending on a numerical solution. 
If it is satisfied on a certain time interval, then, based on a conditional regularity result or blow-up criterion, it excludes a blow-up, and hence guarantees existence of a strong solution on the respective interval. This might be extended to global-in-time existence in certain cases. 
Indeed, for zero force term $\bf \equiv 0$, the  structure theorem by  Leray~\cite{Leray1934}, see also~\cite[Thm.~6.3]{Galdi2000}, states that there is a time {$T_r>0$} such that if a strong solution exists until time {$T_r$}, then it exists for all times. An upper bound for {$T_r$} can be found in~\cite[Thm.~8 (ii)]{Heywood1980}, see also \cite[Rmk.~6.3]{Galdi2000}.

A posteriori existence results have first emerged in the 80s~\cite{BonaPritchardScott1980,HeywoodRannacher1986} for external force term $\bf \equiv 0$, and have been extended to nonzero force  and further developed in \cite{ChernyshenkoConstantinRobinsonEtAl2007,DashtiRobinson2008,MorosiPizzocchero2008}.  
Furthermore, in~\cite{ChernyshenkoConstantinRobinsonEtAl2007,DashtiRobinson2008} the authors show that if the solution is strong, then for sufficiently small discretisation parameters the corresponding numerical solutions satisfy the criterion. 

However, in previous works, the blow-up criterion employed is weaker in the sense that it is based on stronger norms $L^{\infty}(W^{1,2})$, and the data is assumed to be at least  $u_0 \in W^{1,2}_{\div}$ and $\bf \in L^2(L^2) \cap L^1(W^{1,2})$. 
Additionally, the numerical solution has to be rather regular, with the weakest requirement that the numerical solution satisfies $\hat{\bu} \in L^4(W^{1,2}) \cap L^2(W^{2,2})$, see~\cite[Thm.~6.1]{DashtiRobinson2008}. 
This may be viable for spectral methods, but is out of reach for standard finite element methods since they are not twice weakly differentiable. 
Additionally, bounds on the residual $\hat{\bu}$ in $L^1(W^{1,2})\cap L^2(L^2)$ appear in the criterion. 
Subsequent contributions~\cite{MorosiPizzocchero2012, MorosiPizzocchero2015} further estimate the residuals, and thereby make the conditions fully computable for strong solutions which are at least in $C^0(W^{5,2}) \cap  C^1(W^{3,2})$.

Only the first contribution~\cite{BonaPritchardScott1980} by Bona, Pritchard and Scott works with the Serrin criterion in an $L^p$-setting for $p>3$ rather than in a Hilbert setting, where Fourier techniques would be available. 
Indeed, strong solutions $\bu \in L^{\infty}(L^p)$ for $p > 3$ are considered. 
The numerical solution is only required to be in $L^\infty(W^{1,2})$, which is feasible for finite element solutions. 
The a posteriori criterion itself includes $\norm{\bu_0}_{W^{1,2}}$ as well as $\norm{\hat{\bu}}_{L^\infty(L^p)}$, for $p \in (3,6)$, but does not require information on the residual. 
However, let us remark that the result assumes that $\bf = 0$ and that certain a priori error estimates on $\bu$ in $L^2$ are satisfied. 

We aim to extend this to the critical exponent $p = 3$, to nonzero $\bf$, and we do not assume any a priori error estimates or any further properties on $\bu$. 
\smallskip 

\textbf{Objective.} Here we present the first a posteriori existence result in the critical norm $L^{\infty}(L^3)$, based on the blow-up criterion~\cite{AlbrittonBarker2020} in the spirit of~\cite{IskauriazaSereginShverak2003}; see Section~\ref{sec:main} for the main result. 
Due to the scaling properties of the Navier--Stokes equations this is the blow-up criterion in the scaling-optimal norm. 
In particular, our requirements on the initial data and the bounds on the numerical solutions are considerably weaker than in previous work: 
We only require that $\bu_0 \in W^{1,2}_{\div}$, $\bf \in L^2(L^2)$ and we require that the numerical solution $ \hat{\bu}$ is in $L^\infty(W^{1,2}) \cap L^4(L^6) $, but no higher order norms in space are used. 
Furthermore, we need the residual to be in $L^2(L^2)$, but  our criterion only includes  weak norms on the residual, namely in $L^2(W^{-1,2}) \cap L^{3}(W^{-1,3})$, which we expect to lead to less restrictive conditions than the ones in the previous contributions. 
As for the previous works, if the criterion is not met, no conclusion can be drawn, and our criterion is still far from optimal. 
 
In order to not obscure the underlying ideas we shall limit the presentation to the simplest possible case, by handling only periodic boundary conditions, and by using a conforming mixed finite element method in space 
%the lowest order Taylor--Hood element in space
 and an implicit Euler method in time. 
Note however, that one may use any inf-sup stable finite element pair for the space discretisation. 
\smallskip 

\textbf{Proof strategy.} Since the basic idea of a posteriori existence results is not widely established, let us sketch the proof strategy on a simple example.  

\textit{Illustrative ODE example.}
Let us consider an ODE of the form 
\begin{align}\label{eq:ODE}
	y'(t) = y^2(t) \qquad \text{ for } t > 0,
\end{align}
subject to an initial condition $ y(0) = y_0$ 
for some given $y_0 > 0$. 
This ODE has the solution $y(t) = (y_0^{-1} - t)^{-1}$, and blow-up occurs at time $t = y_0^{-1}$. 
To demonstrate the approach, let us forget about this fact, and consider $y$ to be our (unknown) solution, of which we do not know, whether a blow-up occurs on a given interval. 
For simplicity we use the implicit Euler scheme to approximate the solution with a time grid $\{t_i\colon i \in \mathbb{N}_0\}$ with $t_0 = 0$ and time step size $\tau_i \coloneqq t_{i} - t_{i-1}>0$, for $i\in \mathbb{N}$, and $\tau \coloneqq \sup_{i \in \mathbb{N}} \tau_i$. 
Starting from the initial datum $y_0$, the values $y_i$ are given by 
\begin{align}\label{eq:Euler}
	y_{i} = y_{i-1} + \tau_i y_{i}^2 \qquad \text{ for } i \in \mathbb{N}. 
\end{align} 
Notably, for solvability $\tau_i$ has to be sufficiently small depending on $y_{i-1}$ and the time grid points might converge to a finite time point. 
Let $\hat{y}_\tau$ denote the continuous, piecewise affine interpolation of $(t_i, y_i)_{i \in \{0, \ldots, m\}}$, i.e., $\hat{y}_{\tau}(t_i) = y_i$ and $\hat{y}_{\tau}'(t)|_{(t_{i-1},t_i)} = \tfrac{1}{\tau} (y_i - y_{i-1})$. 
We restrict ourselves to the time interval $I = (0,T)$ on which $\hat{y}_{\tau}$ exists, and clearly beyond which no claims can be made. 

Obviously, $\hat{y}_{\tau}$ does not satisfy the ODE~\eqref{eq:ODE}. 
But there is a function $r[\hat{y}_\tau] \coloneqq  \hat{y}_{\tau}' - \hat{y}_{\tau}^2$ on $I$, referred to as residual of $\hat{y}_{\tau}$, such that $\hat{y}_\tau$ satisfies a perturbed version of~\eqref{eq:ODE} 
\begin{align}\label{eq:ODE-pert}
\hat{y}_{\tau}'(t) = \hat{y}_{\tau}^2(t) + r[\hat{y}_{\tau}] \qquad \text{ for a.e. } t \in I.
\end{align}
Here we are in the situation, that the residual is a piecewise quadratic function and one can show that 
\begin{align*}
	r[\hat{y}_{\tau}] \to 0 \quad \text{ as } \tau \to 0,
\end{align*}
Notably, even if we do not know the solution $y$, both $\hat{y}_{\tau}$ as well as the residual $r[\hat{y}_{\tau}]$ can be computed, once the numerical solution is available.

\textit{Step 1 (stability estimate):}   
Taking the difference of~\eqref{eq:ODE} and~\eqref{eq:ODE-pert}, multiplying with the error function $e_{\tau} \coloneqq y - \hat{y}_\tau$ and integrating in $(0,t)$, one can show that 
\begin{align}\label{est:stab-ODE}
\abs{e(t)}^2 
 \leq 
 \int_{0}^t \abs{r[\hat{y}_\tau]}^2 \, \mathrm{d}s 
 +  \int_{0}^t (4 \abs{\hat{y}_\tau(s)} + 1) \abs{e_{\tau}(s)}^2 \, \mathrm{d}s
 + 2 \int_0^t  \abs{e_\tau(s)}^3 \, \mathrm{d}s. 
\end{align}
The important property of this estimate is, that the right-hand side only depends on $e_\tau$ and on $\hat{y}_{\tau}$, but not on $y_{\tau}$. 

\textit{Step 2 (conditional stability):} 
Setting 
\begin{align}\label{def:AM}
	A\coloneqq  \int_{I} \abs{r[\hat{y}_\tau]}^2 \, \mathrm{d}s  \qquad \text{ and } \qquad M \coloneqq \exp \left( \int_{I} (4 \abs{\hat{y}_{\tau}}+ 1) \, \mathrm{d}s  \right), 
\end{align}
a \emph{generalised Gronwall Lemma}~\cite[Lem.~2.1]{Bartels2005} or \cite[Lem.~3.4]{BrunkGiesselmannLukacovaMedvidova2025} allows one to show the following: If the condition 
\begin{align}\label{est:cond-ODE}
	8 (1 + T) (8 A M)^{1/2} \leq 1
\end{align}
is satisfied, then one has that
\begin{align}\label{est:cond-stab-ODE}
	\norm{e_{\tau}}_{L^{\infty}(I)}^2 \leq 2 A M. 
\end{align}

\textit{Step 3 (conditional existence statement):}
As a consequence, we may conclude that $y \in L^{\infty}(I)$, since 
\begin{align}
	\norm{y}_{L^{\infty}(I)}^2 \leq 2 \norm{y - \hat{y}_{\tau}}_{L^{\infty}(I)}^2 + \norm{\hat{y}_{\tau}}_{L^{\infty}(I)}^2 \leq  4 AM  +\norm{\hat{y}_{\tau}}_{L^{\infty}(I)}^2 \leq  c_{\tau}   < \infty. 
\end{align} 
Thus, if condition~\eqref{est:cond-ODE} is satisfied, then $y$ has no blow-up on $I$. 
Since $f(v) = v^2$ is locally Lipschitz continuous, by classical ODE theory, it follows that the solution $y$ exists on $I$. 

\textit{4. Step: (verification of condition):}
It remains to discuss the requirements on the numerical solution $\hat{y}_{\tau}$. 
If $\hat{y}_{\tau}$ exists on $I$, then the terms $A = A(r[\hat{y}_\tau])$ and $M = M(\hat{y}_\tau)$ in~\eqref{def:AM} can be evaluated without further effort, since they only contain computable quantities. 
Then, once $\hat{y}_{\tau}$ is available we may check condition~\eqref{est:cond-ODE}, and if it holds for some time grid, then existence of $y$ on $I$ is verified. 
Note that~\eqref{est:cond-ODE} can be interpreted as a concrete smallness condition on $T$. 

\smallskip 

\textit{Additional challenges for PDEs.}
To apply such a strategy to the Navier--Stokes equations, extra effort has to be invested in the steps above. 
Firstly, the seminal work~\cite{IskauriazaSereginShverak2003} establishes a blow-up criterion that can be used to replace the arguments in \emph{Step 3}. 
Roughly speaking, this criterion states, that if a strong solution stops to exist, then the $L^3$-norm blows up in time. 
Hence, it represents a conditional regularity result, see Proposition~\ref{prop:blowup} for the precise statement. 

For this purpose, we prove a stability result (\emph{Step 1}) and conditional stability results (\emph{Step 2}) in $L^3$-norms for the error $\bu - \hat{\bu}$, where $\bu$ is a Leray--Hopf solution, and $\hat{\bu}$ is a Leray--Hopf solution to a perturbed problem. This uses techniques similar to the ones in \cite{Kucera2013}.  
In those estimates, certain embedding constants and stability constants of the Leray projection appear and must be bounded above in order to obtain a computable criterion resembling~\eqref{est:cond-ODE}. 
One may apply the same type of Gronwall Lemma, but apart from norms on $\hat{\bu}$ an additional initial spatial error appears, and the residuals $\mathbf{r} = \mathbf{r}[\hat{\bu}]$ appear in certain negative Sobolev norms.  

	For the {numerical approximation} of solutions to the incompressible Navier--Stokes equations a variety of space and time discretisations are available.  
	Discrete solvability, stability and approximation properties are available in many cases, see~\cite{John2016,BernardiGiraultHechtEtAl2024} for a review on finite element methods. 
	As in the ODE example, well-posedness of the discrete problem has to be addressed, and one can only deal with time intervals, on which the numerical solution exist. 
	When using, e.g., a time stepping scheme, one may again consider a time interpolation.  
	Unfortunately, this interpolation cannot directly be used as $\hat{\bu}$, but instead one has to work with a reconstruction thereof in space. 
Similarly to the elliptic reconstruction~\cite{MakridakisNochetto2003}, the reconstruction is obtained by considering a certain auxiliary Stokes problem as in~\cite{KarakatsaniMakridakis2007}. 
The error between numerical solution and its reconstruction can be bounded employing certain a posteriori error estimates as in~\cite{Verfurth2013}. 
Those estimates provide computable bounds that can be used to verify~\eqref{est:cond-ODE} in \emph{Step 4}. 
Let us remark that the reconstruction need not be computed, but merely serves as a theoretical tool. 
Additionally, negative norms of the residuals $\mathbf{r}[\hat{\bu}]$ have to be estimated, which is again done by means of a~posteriori error estimates.  

Our approach can be used directly as an error estimator for the Navier--Stokes equations, although in non-standard norms. 
For the Navier--Stokes equations in two space dimensions, a posteriori error estimators were obtained in \cite{Nassreddine17,Nassreddine23}, which provide two-sided error bounds in terms of the estimator. 
In three space dimensions, \cite{Bernardi14} developed an error estimator for the instationary Stokes equations. 
It is important to note that all those results rely on stability arguments in Hilbert spaces, specifically in the $L^2$ framework.
\smallskip 

\textbf{Scope and Outline.} 
Our main results can be found in Section~\ref{sec:main}, and Section~\ref{sec:prelim} introduces some preliminaries. 
Then, in Section~\ref{sec:stab} we deduce a conditional stability result for the incompressible Navier--Stokes equation in three space dimensions based on combined $L^\infty(L^2)$ and $L^\infty(L^3)$ stability estimates using a conditional Gronwall argument. 
%This can be found in Section~\ref{sec:stab}. 
In turn, this stability result in conjunction with a blow-up criterion is used to show conditional existence under suitable conditions on a (reconstruction of a) numerical solution. This constitutes our first main result, and is proved in Section~\ref{sec:main-proof}. 
The construction of a suitable numerical solution and its reconstruction is presented in Section~\ref{sec:num-scheme}. 
Using certain a posteriori estimates we obtain, that the existence criterion is fully computable, which proves our second main result Theorem~\ref{thm:main-2}.  
This allows us to verify the conditional stability estimate a posteriori using the numerical solutions and if satisfied allows us to rule out blow-up in $L^\infty(L^3)$.
As a byproduct we obtain a conditional a posteriori error estimate in Corollary~\ref{cor:apost}. 
Evidently, such estimates can only be expected to hold under certain conditions, since otherwise they would imply the existence of strong solutions. 
\smallskip 

\textbf{Perspectives.}
As in previous contributions, the exponential dependence in the criterion on $T$ cannot be avoided. 
This means that the approach is not suited to show long-time existence of strong solutions. 
However, it is capable to show existence for possibly short, but verified time intervals. 
Still, with efficient implementation and high performance computing~\cite{DeparisGrandperrinQuarteroni2014,KronbichlerDiagneHolmgren2018,PiatkowskiMuethingBastian2018}, it may be possible to extend the time intervals accordingly. 
In general, our estimates are by no means sharp, and can most likely be refined and improved in various ways. 
For example, the fact that we work with a stability estimate in both $L^2$- and $L^3$-norms suggests that this is far from optimal. 
To obtain a fully computable criterion,  embedding constants, as presented in~\cite{MizuguchiTanakaSekineEtAl2017}, and stability constants of the Leray projection have to be estimated. 
Since (especially in the $L^p$ setting) not many explicit constants are available in the literature, for the purpose of visualisation we have collected rather crude estimates in Appendix~\ref{sec:app-const}. 
We expect that this can be considerably improved, and we would like to stress that results like the present one would benefit from the availability of sharp or at least better constants.  
Also the Gronwall type argument may be replaced by an iterative argument as presented in~\cite{KyzaMetcalfeWihler2018} to obtain better estimates. 
Furthermore, the results would benefit from improved a posteriori estimates as well as better ways to compute negative Sobolev norms.  

We expect that our results can be generalised, e.g., to more general domains and boundary conditions. Also better space and time discretisations as well as the use of adaptivity may reduce the computational effort needed to satisfy the existence criterion. 

Furthermore, it would be useful to show, that if a strong solution to the Navier--Stokes equations exists, that then the residual converges to zero in the discretisation parameters.
Such a result would ensure verifiability in the sense that the criterion is indeed satisfied for sufficiently small discretisation parameters, provided that a strong solution exists. 
However, this would require error estimates in non-standard norms, and for this reason is beyond the scope of this work.

Existence results for PDEs  of a posteriori type have not received a lot of attention so far.  
Let us mention, that for monotone elliptic problems a result of this kind was established in~\cite{Ortner2009}, see also the thorough contextualisation therein. 
Furthermore, in~\cite{MizuguchiTakayasuKuboEtAl2017a} an a posteriori existence of global solutions to semilinear parabolic equations was proved. 
There are also recent results on Keller--Segel models \cite{GiesselmannKwon2024, GiesselmannHoffmann}.

Various PDEs exhibit possible blow-up or potentially short-time existence of strong solutions. 
Examples include nonlinear parabolic and wave equations \cite{Ball1977}, Ginzburg--Landau models \cite{MasmoudiZaag2008}, chemotaxis models including the Keller--Segel equations \cite{HerreroVelazquez1997},
 cross diffusion systems 
 \cite{CarrilloHittmeirJuengel2012, BurgerLaurencotTrescases2021} and nonlinear Schrödinger equations \cite{TaoVisanZhang2007}, to mention just a few. 
It would be of interest to find out, whether similar a posteriori existence results can be established in those situations. 
Since conditional regularity results are at the core, this is a strong motivation to pursue them. 

\smallskip 

%---------------------------------

\section{Main results}\label{sec:main}

In this section, we present our main results. 
We compare two Leray--Hopf solutions ${\bu}$,$\hat{\bu}$ to the incompressible Navier--Stokes equations to data $(\bu_0, \bf)$,  $(\hat{\bu}_0, \hat \bf)$ on the three-dimensional torus $\dom$, respectively. 
On  $\hat{\bu}$ we assume additional smoothness. 
If we have some additional regularity properties on  $\r \coloneqq \hat{\bf} -{\bf}$ and a certain condition ensuring stability is satisfied, then one can show that $\bu$ is a strong solution on a certain time interval.

\begin{theorem}\label{thm:main-1}
	Let $T>0$ and $\nu>0$ be given. 
	Let $\bu$ be a Leray--Hopf weak solution to the Navier--Stokes equation to the data $(\bu_0,\bf) \in  W^{1,2}_{\div}(\dom) \times  L^2(0,T;L^2(\dom)^3)$.  
	For a function $\hat \bu \in C([0,T];W^{1,2}_{\div}(\dom))$ let $\hat \bu_0 \coloneqq \hat \bu(0)$ and let $\mathbf{r} = \mathbf{r}[\hat \bu]$ be its residual, such that $\hat \bu$ is a Leray--Hopf solution to the data $(\hat \bu_0,\bf + \mathbf{r})$.  
	Assume that 
	\begin{enumerate}[label = (\roman*)]
		\item \label{itm:main-res}
		$\mathbf{r} = \mathbf{r}[\hat \bu]  \in L^2(0,T;L^2(\dom)^3)$,
		\item \label{itm:main-cond} 
		condition~\eqref{eq:agg} holds with $\bg  \coloneqq -\mathbf{r}$ and $\e_0 \coloneqq \bu_0 - \hat \bu_0$ on $[0,T']$ for some $T' \in (0,T]$.  
	\end{enumerate}
	Then, $\bu$ is a strong solution to the Navier--Stokes equation on $[0,T']$. 
\end{theorem}

This result is based on a conditional stability result and a blow-up criterion in $L^3$ available, which is available in the literature. We prove Theorem~\ref{thm:main-1} in Section~\ref{sec:main-proof}. 

In Section~\ref{sec:num-scheme} we shall present one specific choice of a numerical scheme leading to a numerical solution $\bu_{\tau h}$ and the definition of a (theoretical) $\hat{\bu}$, which satisfies the conditions in Theorem~\ref{thm:main-1} including condition~\ref{itm:main-res}.  
We will explain how condition~\ref{itm:main-cond} can be checked without any knowledge of $\bu$ and without exact knowledge of $\hat{\bu}$, purely using properties of the numerical solution $\bu_{\tau h}$. 
This will constructively prove the following theorem:

\begin{theorem}\label{thm:main-2}
There is a numerical scheme, such that for a certain time $T>0$ the  solution $\bu_{\tau h}$ is in $L^\infty(0,T;W^{1,2}(\dom))$ and it has a modification  $\hat{\bu}$ as in Theorem~\ref{thm:main-1}, such that all terms occurring in~\eqref{eq:agg} can be estimated without using $\bu$, but only using the discrete solution $\bu_{\tau h}$ and the data. 

This makes the condition~\eqref{eq:agg} fully computable, and thus allows to show that $\bu$ is a strong solution on the time interval, on which the condition holds. 
\end{theorem}

% ----------------------------------------
 
\section{Preliminaries}
\label{sec:prelim}

We consider the incompressible Navier--Stokes equations for the velocity $\bu$ and the pressure $\pi$
\begin{align}\label{eq:u}
	\dt\bu &+ (\bu\cdot\nabla)\bu - \nu \Delta\bu +\nabla \pi  = \bf, 
	\qquad  
	\div \bu= 0 \qquad \text{ on }  [0,T]\times \dom,
\end{align}
for given time interval $[0,T]$ and space domain $\dom$, with given periodic function $\bf$ and constant $\nu >0$. 
The system of equations is supplemented by periodic boundary conditions and initial conditions $\bu(0,\cdot) = \bu_0$ for some periodic initial velocity $\bu_0$. 

This section collects the notation, results on the Helmholtz projection, the notions of solutions and  regularity results available for the Navier--Stokes equations as well as tools used in the sequel. 

\subsubsection*{Notation}
For  $p \in [1,\infty]$ and $k \in \mathbb{N}$ we denote the standard Lebesgue and Sobolev spaces of periodic functions on $\dom = [0,1]^3$  by $L^p(\dom)$ and $W^{k,p}(\dom)$ and we abbreviate the norms by $\norm{\cdot}_{L^p}$ and $\norm{\cdot}_{W^{k,p}}$, respectively. 
More precisely, they are given by 
\begin{align*}
	\norm{v}_{L^p} &\coloneqq \begin{cases}
		\left(\int_{\dom} \abs{v(x)}^p \dx\right)^{1/p}  \quad \;\quad & \text{ if } p \in [1,\infty),\\
		\esssup_{x \in \dom} \abs{v(x)}& \text{ if } p =\infty,
	\end{cases}	\\
		\norm{v}_{W^{1,p}} &\coloneqq \begin{cases}
		\left(\norm{v}_{L^p}^p + \norm{\nabla v}_{L^p}^p \right)^{1/p}  \quad & \text{ if } p \in [1,\infty),\\
	\max(\norm{v}_{L^{\infty}} , \norm{\nabla v}_{L^{\infty}})
		& \text{ if } p =\infty.
	\end{cases}	
\end{align*}
Here, we use the convention $\norm{\nabla v}_{L^p} = \norm{\abs{\nabla v}}_{L^p}$ with $\abs{\cdot}$ the Euclidean norm on $\setR^d$. 
The integral mean value free subspaces of $L^p(\dom)$ and $W^{1,p}(\dom)$ are denoted by $L^p_\sim(\dom)$ and $W^{1,p}_\sim(\dom)$, respectively. 
For a Banach space $X$ of scalar functions, the space of vector-valued functions is denoted by $X^d$, for $d \in \mathbb{N}$.
For the vector-valued Sobolev space $W^{k,p}(\dom)^d$, for $d \in \mathbb{N}$, the notion of the norms extend naturally, with $\abs{\cdot}$ denoting the Frobenius norm for matrix-valued arguments. 
Moreover, the divergence-free (solenoidal) subspaces of $L^p(\dom)^3$ and $W^{1,p}(\dom)^3$ are denoted by $L^p_{\div}(\dom)$ and $W^{1,p}_{\div}(\dom)$, respectively.  
Here solenoidality holds in the sense of distributions. 

We denote the dual space  $(W^{1,p}(\dom))' = W^{-1,p'}(\dom)$ with $\frac{1}{p} + \frac{1}{p'} = 1$ and denote the dual norm by $\norm{\cdot}_{W^{-1,p'}}$. 
Similarly, $W^{-1,2}_{\div}(\dom)$ denotes the dual space of $W^{1,2}_{\div}(\dom)$. 
We denote the integral by $\la f,g \ra\coloneqq\int_{\dom} fg \dx$. 
Finally for $p \in [1,\infty]$, an interval $[0,T]$ and a Banach space $X$ we denote the Bochner space of $X$-valued functions by $L^p(0,T;X)$ with norm $\norm{\cdot}_{L^p(0,T;X)}$. 
If the time interval is clear from the context, we abbreviate the norm also by $\norm{\cdot}_{L^p(X)}$.

\subsection{Helmholtz decomposition}
\label{sec:Helmholtz-proj}

The Helmholtz decomposition of $L^2(\dom)^3$ is given as 
\begin{align*}
	L^2(\dom)^3  = L^2_{\div}(\dom) \oplus \nabla W^{1,2}_{\sim}(\dom), 
\end{align*} 
see, e.g.,~\cite[Thm.~III.1.1]{Galdi2011}, where the orthogonality is with respect to the $L^2$-inner product. 
This decomposition is induced by the so-called \textit{Helmholtz} or \textit{Leray projection}, the $L^2$-orthogonal projection $\Pi \colon L^2(\dom)^3 \to L^2_{\div}(\dom)$, defined by 
\begin{align}\label{eq:L2-proj}
	\skp{\Pi \bv}{\bw} = \skp{\bv}{\bw} \quad \text{ for all } \bw \in L^2_{\div}(\dom),
\end{align} 
for $\bv \in L^2(\dom)^3$. 
This projection can be represented as 
\begin{align}\label{eq:L2proj-repr}
	\Pi \bv  = \bv - \nabla \psi,
\end{align}
where $\psi \in W^{1,2}_{\sim}(\dom)$ is the unique solution to the following Poisson problem 
\begin{align}\label{eq:aux-Poisson}
	\skp{\nabla \psi}{\nabla \phi}  = \skp{\bv}{\nabla \phi} \quad \text{ for any } \phi \in W^{1,2}_{\sim}(\dom). 
\end{align}

Let us collect some well-known properties of the Helmholtz projection. 

\begin{lemma}\label{lem:Leray}
	The Helmholtz projection $\Pi \colon L^2(\dom)^3 \to L^2_{\div}(\dom)$ satisfies the following:
	\begin{enumerate}[label = (\roman*)]		
		\item 
		\label{itm:Leray-commutation}
		$\Pi$ commutes with space derivatives, in the sense that one has
		\begin{align*}
			\Pi(\partial_{x_i}\bv) = \partial_{x_i}\Pi(\bv)\quad \text{ for } i\in\{1,2,3\}, \text{ for any } \bv\in W^{1,2}(\dom)^3;
		\end{align*}
		\item 
		\label{itm:Leray-Lp} 
		For any $p \in (1,\infty)$ the projection $\Pi$ is a bounded linear operator mapping $L^{p}(\dom)^3 \to L^{p}(\dom)^3 $. 
		This means, there exists a constant $c = c(p)>0$ such that 
		\begin{align*}
			\norm{\Pi \bv}_{L^p} \leq c  \norm{ \bv}_{L^p} \qquad \text{ for any } \bv \in L^{p}(\dom)^3. 
		\end{align*}
		\item 
		\label{itm:Leray-W1p}
			For any $p \in (1,\infty)$ the projection $\Pi$ is a bounded linear operator from $W^{1,p}(\dom)^3 \to W^{1,p}(\dom)^3 $. 
			I.e., there is a constant $\tilde{c} = \tilde{c}(p)>0$ such that
		\begin{align*}
			\norm{\Pi \bv}_{W^{1,p}} \leq  \tilde{c} \norm{  \bv}_{W^{1,p}} \quad \text{ for any } \bv \in W^{1,p}(\dom)^3. 
		\end{align*}
	\end{enumerate}
	By the stability properties the Helmholtz projection $\Pi$ can be extended to $L^{p}(\dom)^3$ and also to $W^{-1,p'}(\dom)^3$ for any $p \in (1,\infty)$. 
\end{lemma}
\begin{proof} 
	For~\ref{itm:Leray-commutation} we refer to \cite[Lem.~2.9]{RobinsonBook}. 
	For~\ref{itm:Leray-Lp} we refer to \cite[Sec.~III.1]{Galdi2011} and \cite[Thm.~2.28]{RobinsonBook}, see also \cite{FujiwaraMorimoto1977}. 
	The stability in $W^{1,p}(\dom)^3$ is related to elliptic regularity properties, and stated, e.g., in \cite[Thm.~38]{Wichmann2024}. 
\end{proof}

The proof of the stability results is given in Appendix~\ref{app:sec:stab} in order to specify the constants. 

\subsection{Notions of solutions and regularity}
To abbreviate the convective term we use the trilinear form 
\begin{align*}
	b(\bu,\bv,\w) \coloneqq 
	  -  \skp{\bu \otimes \bv}{\nabla \w}  = - \sum_{i,j = 1}^3 u_i v_j \partial_i w_j,
\end{align*}
for $\bu, \bv, \w \in C^{\infty}(\dom)^3$. 
By density and boundedness of the trilinear form this can be extended to certain Sobolev and Lebesgue functions. 
Integrating by parts we have
\begin{align*}
	b(\bu,\bv,\w)  = 
	- \skp{\bu \otimes \bv}{\nabla \w} 
	= \skp{\bu \otimes \w}{\nabla \bv}  + \skp{ \div (\bu) \,\bv}{\w}, 
\end{align*}
and hence $b(\bu,\cdot,\cdot)$ is skew-symmetric for any divergence-free function $\bu \in C^{\infty}(\dom)^3$ meaning that  one has that
\begin{align}\label{eq:conv-skewsym}
	b(\bu,\bv,\w) = - b(\bu,\w,\bv) \qquad \text{ for any }\bv, \bw \in C^\infty(\dom)^3. 
\end{align}

Let us introduce the notion of a weak solution named after Leray and Hopf, see e.g.,~\cite{Seregin2015}. 

\begin{definition}[Leray--Hopf solution]\label{defn:weak}
	For given $T>0$, given $\bf \in L^2(0,T;W^{-1,2}(\dom)^3)$, and given initial velocity $\bu_0 \in L^2_{\div}(\dom)$ we  call a function \begin{align*}
		\bu \in L^\infty(0,T;L^2(\dom)^3) \cap L^2(0,T;W^{1,2}_{\div}(\dom))
	\end{align*}
	 a \emph{weak solution} to the Navier--Stokes equations \eqref{eq:u}, if it satisfies  
		\begin{align}\label{eq:ns-weak}
			- \int_{0}^T  \skp{ \bu}{\partial_t \bv}  \, \mathrm{d}t + \int_{0}^T  b(\bu,\bu,\bv) \, \mathrm{d}t + \nu \int_{0}^T   \skp{\nabla \bu}{\nabla \bv} \, \mathrm{d}t
			&= \int_{0}^T \skp{\bf}{\bv}_{W^{-1,2}\times W^{1,2}} \, \mathrm{d}t,
		\end{align}
		for all $\bv \in C^{\infty}_c((0,T)\times \dom)^3$ 	with $\div \bv = 0$, and if $\bu \in C_w([0,T];L^2(\dom)^3)$ and  
		\begin{align*} 
			\lim_{t \to 0_+}\norm{\bu(t) - \bu_0}_{L^2} = 0,
		\end{align*} 
		and if the following energy inequality is satisfied 
		\begin{align*}
\frac{1}{2} \norm{\bu(t)}^2_{L^2} + \nu \int_{0}^t \norm{\nabla \bu}_{L^2}^2 \, \mathrm{d}s  
 \leq \int_{0}^t \skp{\bf}{\bu}_{W^{-1,2}\times W^{1,2}} \, \mathrm{d}s +  \frac{1}{2}\norm{\bu_0}_{L^2}^2 \quad \text{ for all } t \in [0,T]. 
		\end{align*}
\end{definition}

\begin{definition}[Strong solution]\label{defn:strong}
For given $T>0$, for given function $\bf \in L^2(0,T;L^2(\dom)^3)$, and for given initial velocity $\bu_0 \in W^{1,2}_{\div}(\dom)$ we  call a Leray--Hopf solution $\bu$ to the Navier--Stokes equations a \emph{strong solution} on $[0,T)$, if it enjoys the additional regularity 
\begin{align*}
	\bu \in L^\infty(0,T;W^{1,2}_{\div}(\dom)).
\end{align*}
Note that this implies also that $\bu \in L^2(0,T;W^{2,2}(\dom)^3)$. 
\end{definition}

The following proposition summarises well-known existence results for weak and strong solutions to the Navier--Stokes equations.

\begin{proposition}[Existence of solutions {\cite[Thm.~4.4, 6.4, 6.8, Lem.~8.16]{RobinsonBook}}]\label{thm:NSsols-ex} \hfill \\
	Let $\nu>0$ be a constant, and let $T>0$ be given. 
	\begin{enumerate}[label = (\roman*)]
		\item \label{itm:ex-weak}
		 For any $\bu_0\in L^2_{\div}(\dom)$ and any $\bf\in L^2(0,T;W^{-1,2}(\dom)^3)$ there exists at least one Leray--Hopf weak solution $\bu$ to the Navier--Stokes equation, as in Definition \ref{defn:weak}. 
		
		\item \label{itm:ex-strong}
		For any $\bu_0\in W^{1,2}_{\div}(\dom)$ and any $\bf\in L^2(0,T;L^2(\dom)^3)$ there exists a $T^*\in (0,T]$  such that $\bu$ is the unique strong solution on $[0,T^*)$, in the sense of Definition~\ref{defn:strong}. 
		The time $T^*$ is bounded away from $0$ by a constant depending only on the data $\norm{\bu_0}_{W^{1,2}},\norm{\bf}_{L^2(0,T;L^2)}$ and $\nu$. 		
	\end{enumerate}
\end{proposition}

Strong solutions on $[0,T^*)$ are more regular, i.e., every term in the Navier--Stokes equations is well-defined in $L^2(0,T^*;L^2(\dom)^3)$. 
For the case of Dirichlet boundary conditions this is due to \cite{Ladyzenskaja1967}, see also~\cite[Lem.~6.2]{RobinsonBook} for the periodic case and for $\bf=0$. 
The following proposition summarises well-known regularity results for weak and strong solutions to the Navier--Stokes equations. 

\begin{proposition}[Regularity {\cite[Thm.~4.4, 6.4, 6.8, Lem.~8.16]{RobinsonBook}}]\label{prop:NSsols-reg} \hfill\\ 
Let $\nu>0$ be a constant, and let $T>0$. 
For $\bu_0\in L^2_{\div}(\dom)$ and for~$\bf\in L^2(0,T;W^{-1,2}(\dom)^3)$ let $\bu$ be a Leray--Hopf solution to the Navier--Stokes equations. 
\begin{enumerate}[label = (\roman*)]
	\item \label{itm:reg-LH-reg}
	Leray--Hopf solutions enjoy the additional regularity $\dt\bu\in L^{4/3}(0,T;W^{-1,2}_{\div}(\dom))$ and the corresponding pressure $\pi$ is bounded in  $L^{5/3}(0,T;L^{5/3}(\dom))$. 
	\item \label{itm:reg-Serrin}
	 Let~$\bf\in L^2(0,T;L^2(\dom)^3)$. 
	 If the function $\bu$ satisfies the \emph{Ladyzhenskaya--Prodi--Serrin condition} for some $T' \in (0,T]$, namely, 
	\begin{align}
		\quad 	\bu \in L^p(0,T';L^q(\dom)^3) \quad \text{ for some } p\in [2,\infty), \; q\in(3,\infty] \text{ with } \tfrac{2}{p} + \tfrac{3}{q} = 1, \label{eq:serrin}
	\end{align}
	then $\bu$ is unique and strong solution on $[0,T')$. 
    
   \item \label{itm:reg-eq}  
   Let~$\bf\in L^2(0,T;L^2(\dom)^3)$. 
   If the function $\bu$ is a (unique) strong solution on $[0,T')$, then the following regularity holds
   	\begin{align*}
   		\dt\bu, \,(\bu\cdot\nabla)\bu,\, \nabla \pi \in L^2(0,T';L^2(\dom)^3).
   	\end{align*}
   
\end{enumerate}

\end{proposition}

In the following we shall choose $T^*$ maximal, i.e., such that $\bu$ is a strong solution on $[0,T^*) \subset [0,T]$ and it cannot be extended as strong solution beyond $[0,T^*)$.  

In the celebrated work~\cite{IskauriazaSereginShverak2003} for the Navier--Stokes equations on the whole space domain a criterion was given for strong solutions ceasing to exist. 
More specifically, it was proved, that if $t \mapsto \norm{\bu(t)}_{L^3}$ remains bounded, then the solution is smooth.  
We require a version on the torus, as stated in the following proposition. 
It is a consequence of a localised version due to~\cite{AlbrittonBarker2020} of the above mentioned criterion.  
 
\begin{proposition}[$L^3$-criterion  {\cite{IskauriazaSereginShverak2003,AlbrittonBarker2020}}]\label{prop:blowup}
	Let $T>0$, and $\bf \in L^2(0,T;L^2(\dom)^3)$ be given and let $(\bu,\pi)$ be a strong solution to~\eqref{eq:ns-weak} starting from initial data $\bu_0\in W^{1,2}_{\div}(\dom)$ with maximal existence time $T^* \in (0,T]$. 
	Provided that $T^*<T$, one has that
	\begin{align*}
		\limsup_{t\to T^*}\norm{\bu(t)}_{L^3}=\infty.
	\end{align*}
\end{proposition}

% --------------------------------------------------

\section{Stability estimates}\label{sec:stab}

In this section, we show stability estimates for solutions to the incompressible Navier--Stokes equations.
First, we present classical stability results for Leray--Hopf solutions in~$L^2(\dom)$.
Additionally, we obtain a stability estimate in  $L^3(\dom)$, which is inspired by the approach in ~\cite{Kucera2013}. 

We shall combine both to obtain our main stability  result in Proposition~\ref{prop:main-stab}. 

We consider a weak solution~$\bu$ to the data~$\bu_0$ and~$\bf$, see Definition \ref{defn:weak}, and another weak solution~$\hat \bu$ to the data~$\hat \bu_0$ and $\hat \bf$. 
We denote the difference of the velocities by 
\begin{align}\label{def:error}
	\e \coloneqq \bu-\hat\bu, 
	%\quad \text{ and } \quad   q \coloneqq \pi-\hat \pi. 
\end{align}
and for the data we set $\e_0 \coloneqq \bu_0 - \hat \bu_0$ and $\bg \coloneqq \bf - \hat \bf$. 

We aim for stability estimates on~$\e$ for which the right-hand side is allowed to depend on norms of $\hat \bu$ but not on norms of $\bu$. 
This is the typical situation in weak-strong estimates, where only regularity of one of the solutions is used.

\subsection{Auxiliary estimates}
\label{sec:aux}

Before doing so, let us collect some estimates and label the respective constants for further use. 
We require the stability constants  $c_{\Pi1}, c_{\Pi2}>0$ of the Helmholtz projection $\Pi$ according to Lemma~\ref{lem:Leray} for specific exponents, such that 
\begin{alignat}{3}\label{est:stab-LerayA}
	\norm{\Pi \bv}_{L^3} 
	&\leq \cPiLstabA  \norm{ \bv}_{L^3} \qquad &&\text{ for any } \bv \in L^{3}(\dom)^3,\\\label{est:stab-LerayB}
	\norm{\Pi \bv}_{W^{1,3/2}} 
	&\leq \cPiLstabC  \norm{ \bv}_{W^{1,3/2}} \quad &&\text{ for any } \bv \in W^{3/2}(\dom)^{3}.
\end{alignat}
Furthermore, in the remainder of this work we use the embedding inequalities with constants $\cembA,\cembB>0$ such that 
\begin{alignat}{3}\label{est:emb}
	\norm{\bv}_{L^6}&\leq \cembA \norm{\bv}_{W^{1,2}} \quad &&\text{ for any } \bv \in W^{1,2}(\dom)^3,\\
	 \label{est:emb2}
		\norm{\bv}_{L^3}&\leq \cembB \norm{\bv}_{W^{1,3/2}} \quad &&\text{ for any } \bv \in W^{1,3/2}(\dom)^3.
\end{alignat}
Specific, yet not sharp, values for those constants are collected in Appendix~\ref{app:sec:stab}. 

Let us state some auxiliary results needed in the following. 

\begin{lemma}[{\cite[Lem.~2]{Kucera2013}}]\label{lem:visc-term} For any $\bv \in W^{2,2}(\dom)^3$ the following identity holds 
	\begin{align*}
		-	\skp{\Delta \bv}{\bv \abs{\bv}} = \bignorm{\abs{\bv}^{1/2} \nabla \bv}_{L^2}^2 + \tfrac{4}{9} \norm{\nabla (\abs{\bv}^{3/2})}_{L^2}^2. 
	\end{align*}
\end{lemma}

\begin{lemma}\label{lem:Helholz-aux}
	The Helmholtz projection $\Pi$ satisfies the following stability results with constants $\cembA, \cPiLstabA, \cPiLstabC > 0$ as in~\eqref{est:stab-LerayA}--\eqref{est:emb}
	\begin{align*}
		\norm{\Pi(\abs{\bv}\bv)}_{L^3} 
		& \leq c_{\Pi1}  \,\cembA	 
		\norm{\bv}_{L^3}^{1/2}
		\left(\norm{\bv}_{L^3}^{3} 
		+ \tfrac{9}{4} \norm{\abs{\bv}^{1/2} \nabla\bv}_{L^2}^2\right)^{1/2}, \\
		\norm{\Pi(\abs{\bv}\bv)}_{W^{1,3/2}} 
		& \leq \cPiLstabC 
		\left(
		\norm{\bv}_{L^{3}}^2
		+ 2 \norm{\bv}_{L^3}^{1/2}\norm{\abs{\bv}^{1/2} \nabla\bv}_{L^{2}}\right),
	\end{align*}
	for sufficiently smooth functions $\bv$. 
\end{lemma}
\begin{proof}
	For any vector-valued function one has 
	\begin{align*}
		\abs{\nabla(\abs{\bv}^{1/2}\bv)}  \leq  \tfrac{1}{2} \abs{\bv}^{-1/2} \abs{\bv \, \nabla \bv} + \abs{\bv}^{1/2} \abs{ \nabla \bv } 
	\leq \tfrac{3}{2} \abs{\bv}^{1/2} \abs{\nabla \bv}.
	\end{align*}
	Applying this in combination with the $L^3$-stability of the Helmholtz projection in Lemma~\ref{lem:Leray}~\ref{itm:Leray-Lp}, see also~\eqref{est:stab-LerayA}, using Hölder's inequality and  $W^{1,2}(\dom)\hookrightarrow L^6(\dom)$ in \eqref{est:emb} we have
	\begin{equation}\label{est:P3}
		\begin{aligned}
			\norm{\Pi(\abs{\bv}\bv)}_{L^3} 
			&\leq c_{\Pi1} 
			\norm{\abs{\bv}\bv}_{L^3} 
			\leq  c_{\Pi1}  
			\norm{\abs{\bv}^{1/2}}_{L^6}
			\norm{\abs{\bv}^{1/2}\bv}_{L^6}
			\\
			&  \leq c_{\Pi1}\,  \cembA
			\norm{\abs{\bv}^{1/2}}_{L^6} \left(\norm{\abs{\bv}^{1/2}\bv}_{L^2}^2 + \norm{\nabla(\abs{\bv}^{1/2}\bv)}_{L^2}^2\right)^{1/2}
			\\
			& \leq c_{\Pi1}  \,\cembA	 
			\norm{\bv}_{L^3}^{1/2}
			\left(\norm{\bv}_{L^3}^{3} 
			+ \tfrac{9}{4} \norm{\abs{\bv}^{1/2} \nabla\bv}_{L^2}^2\right)^{1/2}.
		\end{aligned}
	\end{equation}
	This proves the first estimate. 
	
	Thanks to the stability of the Helmholtz projection  in Lemma~\ref{lem:Leray}~\ref{itm:Leray-W1p}, see~\eqref{est:stab-LerayB}, we have
	\begin{equation}
	\begin{aligned}\label{est:stabL3-8}
		\norm{\Pi(\abs{\bv}\bv)}_{W^{1,3/2}} 
		& \leq \cPiLstabC 
		\norm{\abs{\bv}\bv}_{W^{1,3/2}} 
		=  \cPiLstabC 
		\left(
		\norm{\abs{\bv}\bv}_{L^{3/2}}^{3/2} 
		+ \norm{\nabla(\abs{\bv}\bv)}_{L^{3/2}}^{3/2} \right)^{2/3}\\
		&\leq 
		 \cPiLstabC 
		\left(
		\norm{\abs{\bv}\bv}_{L^{3/2}} 
		+ \norm{\nabla(\abs{\bv}\bv)}_{L^{3/2}} \right).
	\end{aligned}
	\end{equation}
		One can show that 	
		\begin{align*}
		\abs{\nabla(\abs{\bv}\bv)} \leq  2\abs{\bv} \abs{ \nabla \bv},
	\end{align*} 
	and hence we can estimate
		\begin{align}\label{est:stabL3-9}
		\norm{\nabla(\abs{\bv}\bv)}_{L^{3/2}} 
		\leq 2 \norm{\abs{\bv}\nabla\bv}_{L^{3/2}} 
		\leq  2 \norm{\bv}_{L^3}^{1/2}\norm{\abs{\bv}^{1/2} \nabla\bv}_{L^{2}}.	 
	\end{align}
	Applying this in~\eqref{est:stabL3-8} shows that	
\begin{align}\label{est:stabL3-10}
		\norm{\Pi(\abs{\bv}\bv)}_{W^{1,3/2}} 
		& \leq \cPiLstabC 
		\left(
		\norm{\bv}_{L^{3}}^2
		+  2 \norm{\bv}_{L^3}^{1/2}\norm{\abs{\bv}^{1/2} \nabla\bv}_{L^{2}}\right),
	\end{align}
	which finishes the proof. 
\end{proof}

Let us collect the assumptions on the data.
\begin{assumption}[data]
	\label{ass:data} 
	We assume that the following functions are given

	 \begin{align*}
	 	\bu_0,\hat\bu_0\in W^{1,2}_{\div}(\dom) \quad \text{ and  }
	 	\quad  \bf,\hat\bf\in L^2(0,T;L^2(\dom)^3),
	 	\end{align*}
	 	for $T>0$. 
\end{assumption}

\subsection{Combined $L^2$- and $L^3$-estimate}

Under those assumptions Leray--Hopf weak solutions $\bu, \hat \bu$ to the Navier--Stokes equations for the given data exist on $[0,T]$, see Proposition~\ref{thm:NSsols-ex}~\ref{itm:ex-weak}. 

\begin{lemma}[$L^2$-estimate]\label{lem:l2-stab}
	Let $\nu>0$ be constant and let $T'>0$ be given. Furthermore, let $\bu$ and $\hat \bu$ be Leray--Hopf weak solutions to the data $(\bu_0,\bf)$ and $(\hat \bu_0,\hat  \bf)$ satisfying Assumption~\ref{ass:data} for $T \geq T'$, respectively. 
	Assume that $\hat \u$ is a strong solution on $[0,T')$. 
	
	Then, with $\bg \coloneqq \bf - \hat \bf$, the function $\e \coloneqq \bu - \hat \bu$ satisfies the following estimate
		\begin{align*}
		\tfrac{1}{2} \norm{\e(t)}_{L^2}^2 
		+ \tfrac{\nu}{2} \int_0^t \norm{\nabla \e}_{L^2}^2 \, \mathrm{d}s 
		\leq \,
			\tfrac{1}{2} \norm{\e_0}_{L^2}^2 
			& + 
		\left(\tfrac{1}{2} + \tfrac{1}{\nu}\right)  \int_0^t \norm{\bg}_{W^{-1,2}}^2 \, \mathrm{d}s\\
		& 
		+ \int_0^t \left(\tfrac{1}{2} + \tfrac{2 \cembA}{ \nu}\norm{\hat\bu}_{L^6}^{2} + 2 \tfrac{ 3^3 \cembA^2}{\nu^3}\norm{\hat\bu}_{L^6}^4\right)\norm{\e}_{L^{2}}^2 \, \mathrm{d}s. 
	\end{align*}	
	 for any $t \in [0,T')$, with constant $\cembA>0$ in~\eqref{est:emb}. 
\end{lemma}
\begin{proof}
	The proof follows standard arguments as used in the proof of weak-strong uniqueness, see, e.g.,~\cite{Fefferman_Robinson_Rodrigo_2018}. 
	We present it for the sake of completeness. 

 First note, that $\u$ is a Leray--Hopf solution and thus it satisfies the energy inequality 
 	\begin{align*}
 	\frac{1}{2} \norm{\bu(t)}^2_{L^2} + \nu \int_{0}^t \norm{\nabla \bu}_{L^2}^2 \, \mathrm{d}s  
 	\leq \int_{0}^t \skp{\bf}{\bu} \, \mathrm{d}s +  \frac{1}{2}\norm{\bu_0}_{L^2}^2 \qquad \text{ for all } t \in [0,T]. 
 \end{align*}
 
 Since $\hat \bu$ is a strong solution on $[0,T')$ it satisfies the energy equality for all $t \in [0,T')$. 
 In view of Proposition~\ref{prop:NSsols-reg}~\ref{itm:reg-LH-reg} and Definition~\ref{defn:strong} $\int_0^t \skp{\partial_t \bu}{ \hat \bu}_{W^{-1,2}_{\div}\times W^{1,2}_{\div}} \, \mathrm{d} s$ is well-defined in the sense of duality relations, and hence in the equation \eqref{eq:ns-weak} for $\hat \bu$ we can test with $\bv = \bu$. 
 Here $\skp{a}{b}_{W^{-1,2}_{\div}\times W^{1,2}_{\div}}$ denotes the dual paring between $W^{-1,2}_{\div}(\dom)$ and $W^{1,2}_{\div}(\dom)$
 Also, thanks to Proposition~\ref{prop:NSsols-reg}~\ref{itm:reg-eq} the term  $\int_0^t \skp{\partial_t \hat \bu}{\bu} \, \mathrm{d} s$, so that
 in the equation for $\bu$ we can test with $\bv  = \hat \bu$. 
 
  Combining this with the energy (in)equalities for $\bu$ and $\hat \bu$ yields for any $t \in [0,T')$ that
  \begin{equation}\label{est:error-L2-a}
 \begin{aligned}
 	\frac{1}{2}\norm{\e(t)}_{L^2}^2  + \nu \int_0^t \norm{\nabla \e}_{L^2}^2\, \mathrm{d}s 
 	&- \int_0^t (b(\bu,\bu, \hat \bu) + b(\hat \bu, \hat \bu, \bu))\, \mathrm{d}s \\
 	&\leq \int_{0}^t \skp{\bg}{\e} \, \mathrm{d}s +  \frac{1}{2}\norm{\e_0}_{L^2}^2.
 \end{aligned}
 \end{equation}
 Replacing $\bu = \e + \hat \bu$, and employing the skew-symmetry of the convective term as in \eqref{eq:conv-skewsym} since both $\bu$ and $\hat \bu$ are divergence-free, we obtain 
 \begin{equation}\label{est:L2-stab-1}
 \begin{aligned}
 	-b(\bu,\bu,\hat \bu) - b(\hat \bu, \hat \bu, \bu)  
 	&= - b(\e + \hat \bu, \e + \hat \bu, \hat \bu) 
 	-  b(\hat \bu, \hat \bu, \e + \hat \bu)\\
 	&  = - b(\e,\e,\hat \u) 	- b(\e,\hat \bu, \hat \bu) - b(\hat \u, \e, \hat \bu)  - b(\hat \bu, \hat \bu, \e)\\
 	& = b(\e, \hat \bu, \e)= -\skp{\e \otimes \hat \bu}{\nabla \e}.  
 \end{aligned}
 \end{equation}
Noting that by interpolation and employing the embedding~\eqref{est:emb} we obtain 
 \begin{align*}
 \norm{\e}_{L^3} 
& \leq 
 \norm{\e}_{L^2}^{1/2} \norm{\e}_{L^6}^{1/2} 
 \leq  \cembA^{1/2}
 \norm{\e}_{L^2}^{1/2} \left( \norm{\e}_{L^2}^{2} + \norm{\nabla \e}_{L^2}^2 \right)^{1/4} \\
 &\leq 
  \cembA^{1/2}
 \norm{\e}_{L^2}^{1/2} \left( \norm{\e}_{L^2}^{1/2} + \norm{\nabla \e}_{L^2}^{1/2} \right)
  = 
 \cembA^{1/2}
 \left(
 \norm{\e}_{L^2}  +  \norm{\e}_{L^2}^{1/2} \norm{\nabla \e}_{L^2}^{1/2} \right),
 \end{align*}
we can 
 estimate~\eqref{est:L2-stab-1} with Hölder's and Young's inequality further as 
 \begin{equation}\label{est:L2-stab-2}
	\begin{aligned}
 	\abs{\skp{ \e \otimes \hat\bu}{\nabla \e}} 
 	&\leq \norm{\hat\bu}_{L^6}\norm{\e}_{L^3}\norm{\nabla\e}_{L^2}
 	\\
 	& \leq \cembA^{1/2} \norm{\hat \bu}_{L^6} \left( \norm{\e}_{L^2}\norm{\nabla \e}_{L^2} + \norm{\e}_{L^2}^{1/2} \norm{\nabla \e}_{L^2}^{3/2}\right) \\
 	& \leq \varepsilon \norm{\nabla \e}_{L^2}^2
 	 + \tfrac{1}{2\varepsilon} \cembA \norm{\hat \bu}_{L^6}^2 \norm{\e}_{L^2}^2
 	 + \tfrac{3^3 }{2^5 \varepsilon^3} \cembA^2 \norm{\hat \bu}_{L^6}^4 \norm{\e}_{L^2}^2. 
 \end{aligned}
\end{equation}
	The right-hand side of \eqref{est:error-L2-a} can be further estimated by employing duality of norms and Young's inequality to obtain
 \begin{equation}\label{est:L2-stab-3}
	\begin{aligned}
	\abs{\skp{\bg}{\e}}	
	&\leq 
	\norm{\bg}_{W^{-1,2}} \norm{\e}_{W^{1,2}} 
	\leq \norm{\bg}_{W^{-1,2}}(\norm{\e}_{L^{2}} + \norm{\nabla \e}_{L^{2}}) \\
	&\leq 
	\left(\tfrac{1}{2} + \tfrac{1}{4\varepsilon}\right)  \norm{\bg}_{W^{-1,2}}^2  + \varepsilon \norm{ \nabla \e}_{L^2}^2 + \tfrac{1}{2} \norm{\e}_{L^{2}}^2.
	 \end{aligned}
\end{equation}
Choosing $\varepsilon=\tfrac{\nu}{4}$ and applying \eqref{est:L2-stab-1}--\eqref{est:L2-stab-3} in \eqref{est:error-L2-a} yields 
	\begin{align*}
		\tfrac{1}{2} \norm{\e(t)}_{L^2}^2 
		+ \frac{\nu}{2} \int_0^t \norm{\nabla \e}_{L^2}^2 \, \mathrm{d}s 
		\leq \,
		&	\tfrac{1}{2} \norm{\e_0}_{L^2}^2  + 
		\left(\tfrac{1}{2} + \tfrac{1}{\nu}\right)  \int_0^t \norm{\bg}_{W^{-1,2}}^2 \, \mathrm{d}s\\
		& \quad 
		+ \int_0^t \left(\frac{1}{2} + \frac{2 \cembA}{ \nu}\norm{\hat\bu}_{L^6}^{2} + 2 \frac{ 3^3 \cembA^2}{\nu^3}\norm{\hat\bu}_{L^6}^4\right)\norm{\e}_{L^{2}}^2 \, \mathrm{d}s. 
	\end{align*}	
This proves the claim. 
\end{proof}

Note that the right-hand side of the estimate in Lemma~\ref{lem:l2-stab} is bounded for any strong solution $\hat \bu$.

Now we state and prove the main stability result of this section. 
Note that the result does not require  $\e$ to have zero mean-value. 

\begin{proposition}[combined $L^2$- and $L^3$-estimate]
	\label{prop:main-stab}
	Let $\nu>0$ be a constant, let $T'>0$ and let $\cPiLstabA,\cPiLstabC, \cembA$ be the constants in \eqref{est:stab-LerayA}--\eqref{est:emb}, and let $(\bu_0, \bf)$ and $(\hat \bu_0, \hat \bf)$ be data as in Assumption \ref{ass:data} for $T \geq T'$. 
	Let $\bu$ and $ \hat \bu$ be strong solutions to~\eqref{eq:u} on $[0,T')$  with the respective data, see Definition~\ref{defn:strong}. 
	Denoting $\bg \coloneqq \bf - \hat \bf$, the function  $\e \coloneqq \bu - \hat \bu$ satisfies the following estimate  for all $t\in(0,T')$
%	\begin{align*}
%		\tfrac{1}{3}  \norm{\e(t)}_{L^3}^3
%		&
%		+ \tfrac{1}{2} \norm{\e(t)}_{L^2}^2 
%		+ \tfrac{\nu}{4}
%		\int_{0}^t \bignorm{\abs{\e}^{1/2} \nabla \e}_{L^2}^2
%		+ \norm{\nabla \e}_{L^2}^2
%		+ \tfrac{16}{9} \norm{\nabla (\abs{\e}^{3/2})}_{L^2}^2
%				\ds \\
%		\leq  &
%		\tfrac{1}{3}\norm{\e_0}_{L^3}^3  
%		+  \tfrac{1}{2} \norm{\e_0}_{L^2}^2 
%		+  \cPiLstabC^3\int_{0}^t \left(\tfrac{1}{3} + \tfrac{2^3}{\sqrt{18}\nu^{3/2}}\right)\norm{\bg}_{W^{-1,3}}^3 + \left(\tfrac{1}{2} + \tfrac{1}{\nu}\right)  \norm{\bg}_{W^{-1,2}}^2 \ds
%		\\
%		%%
%		&
%		+ \int_{0}^t \left(\tfrac{1}{2} + \tfrac{2 \cembA}{ \nu}\norm{\hat\bu}_{L^6}^{2} + 2 \tfrac{ 3^3 \cembA^2}{\nu^3}\norm{\hat\bu}_{L^6}^4\right)\norm{\e}_{L^{2}}^2 \ds
%		+  (1+\tfrac{\nu}{9})\int_{0}^t 
%		\norm{\e}_{L^3}^{3} 
%		\ds
%		\\
%		& +   \tfrac{(\cPiLstabA \, \cembA)^2}{\nu}\int_{0}^t  \norm{\e}_{L^3}^{3} 
%		\left(\norm{\e}_{L^3}^2 + \norm{\e}_{L^3}  \norm{\hat \bu}_{L^6}^2 \right)  \ds \\
%		&	+ 
%		\tfrac{9(\cPiLstabA\, \cembA)^2 }{4 \nu}\int_{0}^t \norm{\abs{\e}^{1/2}\nabla \e}_{L^2}^2 \left( 
%		\norm{\e}_{L^3}^2 
%		+ \norm{\e}_{L^3}\norm{ \hat \bu}_{L^{6}}^2
%		\right)  \ds \\
%		&	+ 
%		\tfrac{9\cPiLstabA^2\, \cembA^4}{ 4\nu} \int_{0}^t (\norm{\e}_{L^2}^2 + \norm{\nabla\e}_{L^2}^2)\norm{\e}_{L^3}\norm{ \nabla \hat \bu}_{L^{2}}^2 \ds.
%	\end{align*}
%

	\begin{align*}
		\tfrac{1}{3}  \norm{\e(t)}_{L^3}^3
		&
		+ \tfrac{1}{2} \norm{\e(t)}_{L^2}^2 
		+ \tfrac{\nu}{4}
		\int_{0}^t \bignorm{\abs{\e}^{1/2} \nabla \e}_{L^2}^2
		+ \norm{\nabla \e}_{L^2}^2
		+ \tfrac{16}{9} \norm{\nabla (\abs{\e}^{3/2})}_{L^2}^2
		\ds
		\\
		\leq  &
		\tfrac{1}{3}\norm{\e_0}_{L^3}^3  
		+  \tfrac{1}{2} \norm{\e_0}_{L^2}^2 
		+  \int_{0}^t \tfrac{\cPiLstabC^3}{3}\left(1 + \tfrac{2^4}{\nu^{3/2}}\right)\norm{\bg}_{W^{-1,3}}^3 + \left(\tfrac{1}{2} + \tfrac{1}{\nu}\right)  \norm{\bg}_{W^{-1,2}}^2 \ds
		\\
		&
		+ \int_{0}^t \left(\tfrac{1}{2} + \tfrac{2 \cembA}{ \nu}\norm{\hat\bu}_{L^6}^{2} + 2 \tfrac{ 3^3 \cembA^2}{\nu^3}\norm{\hat\bu}_{L^6}^4\right)\norm{\e}_{L^{2}}^2 \ds
		+  (1+\tfrac{\nu}{9})\int_{0}^t 
		\norm{\e}_{L^3}^{3} 
		\ds
		\\
		& + 
\tfrac{9}{4\nu } \cPiLstabA^2 \cembA^2 
		\int_{0}^t  
		\left(\norm{\e}_{L^3}^2 + \cembA^2 \norm{\e}_{L^3}   \norm{ \hat \bu}_{W^{1,2}}^2 \right)\\
		& \qquad \qquad \qquad\qquad \left( \tfrac{4}{9} \norm{\e}_{L^3}^{3} 
		+  \norm{\e}_{L^2}^2
		+  \norm{\abs{\e}^{1/2}\nabla \e}_{L^2}^2
		+ \norm{\nabla\e}_{L^2}^2
		\right)
		\ds 
	\end{align*}
\end{proposition}
\begin{proof} 
	Since $\hat \bu,\bu$ are strong solutions, see~Definition~\ref{defn:strong}, by Proposition~\ref{prop:NSsols-reg}~\ref{itm:reg-eq} all the following steps are justified. 
	
	While $ \bu, \hat \bu$ are divergence-free	the same is not true in general for $\abs{\e} \e$. 
	For this reason, by~\eqref{eq:L2proj-repr} for some $\psi \in W^{1,2}_{\sim}(\dom)$ we have 
	\begin{align*}
		\abs{\e} \e = \Pi(\abs{\e}\e) + \nabla \psi. 
	\end{align*} 
	Taking the difference of the weak formulation~\eqref{eq:ns-weak} for $(\bu,\pi)$ and for $(\hat \bu, \hat \pi)$, and testing with $\Pi(\abs{\e}\e)$ yields 
	\begin{align}\label{est:stabL3-1a}
		\skp{\partial_t \e}{\Pi(\abs{\e}\e)} 
		& -  \skp{\bu \otimes \bu - \hat \bu \otimes \hat \bu}{ \nabla \Pi(\abs{\e}\e) } + \nu \skp{\nabla \e}{ \nabla \Pi(\abs{\e}\e)}  = -  \skp{\bg}{\Pi(\abs{\e}\e)}.
	\end{align}
Note that the time derivative commutes with the Helmholtz projection and hence we have 
	\begin{align}\label{est:stabL3-1b}
	\skp{\partial_t \e}{\Pi(\abs{\e}\e)} &= \skp{\Pi(\partial_t \e)}{\abs{\e}\e} = \skp{\partial_t \Pi(\e)}{\abs{\e}\e} = \skp{\partial_t \e}{\abs{\e}\e},
\end{align}
since $\e$ is divergence-free.
Furthermore, due to the periodic boundary conditions, by Lemma~\ref{lem:Leray}~\ref{itm:Leray-commutation}, $\Pi$ commutes also with spatial derivatives and thus we find 
%	\begin{align*}
%	\skp{\partial_t \e}{\Pi(\abs{\e}\e)} &= \skp{\Pi(\partial_t \e)}{\abs{\e}\e} = \skp{\partial_t \Pi(\e)}{\abs{\e}\e} = \skp{\partial_t \e}{\abs{\e}\e},\\
%	\skp{\nabla \e}{ \nabla \Pi(\abs{\e}\e)} & = \sum_{i=1}^3 \skp{\partial_i \e}{ \partial_i \Pi(\abs{\e}\e)} = \sum_{i=1}^3 \skp{\partial_i \e}{  \Pi(\partial_i(\abs{\e}\e))} = \sum_{i=1}^3 \skp{\partial_i \Pi\e}{  \partial_i(\abs{\e}\e)} = \skp{\nabla \e}{ \nabla \abs{\e}\e}. 
%\end{align*}
	\begin{align}\label{est:stabL3-1c}
		\skp{\nabla \e}{ \nabla \Pi(\abs{\e}\e)} & =
				\skp{\nabla \Pi \e}{ \nabla (\abs{\e}\e)} 
		 = \skp{\nabla \e}{ \nabla \abs{\e}\e}. 
	\end{align}
Applying \eqref{est:stabL3-1b}--\eqref{est:stabL3-1c} in \eqref{est:stabL3-1a} yields 
	\begin{equation}\label{est:stabL3-1}
		\begin{aligned}
			\skp{\partial_t \e}{\abs{\e}\e}  
			& -  \skp{\bu \otimes \bu - \hat \bu \otimes \hat \bu }{\nabla \Pi(\abs{\e}\e) }   +\nu  \skp{\nabla \e}{ \nabla (\abs{\e}\e)}  = -  \skp{\bg}{\Pi(\abs{\e}\e)} .
		\end{aligned}
	\end{equation}
	On the first term we obtain 
	\begin{align}\label{est:stabL3-2}
			\skp{\partial_t \e}{\abs{\e}\e}   
		= 
		\frac{1}{3} \frac{\mathrm{d}}{\mathrm{d}t} \norm{\e}_{L^3}^3. 
	\end{align}
	On the third term of \eqref{est:stabL3-1} we may use Lemma~\ref{lem:visc-term} to find that 
	\begin{align}\label{est:stabL3-3}
		\nu  \skp{\nabla \e}{ \nabla (\abs{\e}\e)}  
		= \nu \bignorm{\abs{\e}^{1/2} \nabla \e}_{L^2}^2 + \tfrac{4 \nu}{9} \norm{\nabla (\abs{\e}^{3/2})}_{L^2}^2. 
	\end{align}
	It remains to estimate the convective term and the one including $\bg$ in \eqref{est:stabL3-1}. 
Replacing $\bu = \e + \hat \bu$ and employing the skew-symmetry of the convective term we find
	\begin{equation}\label{est:stabL3-4}
		\begin{aligned}
			&\skp{\bu \otimes \bu - \hat \bu \otimes \hat \bu }{\nabla \Pi(\abs{\e}\e) } 
			= 
			 \skp{\e \otimes \e + \hat \bu \otimes \e + \e \otimes \hat \bu}{\nabla \Pi(\abs{\e}\e) } \\
			& \qquad \qquad\qquad = 
			- \skp{\e \otimes\Pi(\abs{\e}\e)}{\nabla \e }  
			-  \skp{\hat \bu \otimes\Pi(\abs{\e}\e)}{\nabla \e } 
			- 	\skp{\e \otimes\Pi(\abs{\e}\e)}{\nabla \hat\bu } \\
			& \qquad \qquad\qquad \eqqcolon (\mathrm{I}) +  (\mathrm{II}) +  (\mathrm{III}). 
		\end{aligned}
	\end{equation}
For the first term, using Hölder's inequality as well as the estimate on $\Pi$ in Lemma~\ref{lem:Helholz-aux} and Young's inequality we obtain 
	\begin{equation}\label{est:stabL3-5}
		\begin{aligned}
			(\mathrm{I}) 
			&\leq
			\norm{\abs{\e}^{1/2} \nabla \e}_{L^2} 
			\norm{\abs{\e}^{1/2}}_{L^6}
			\norm{\Pi(\abs{\e}\e)}_{L^3} \\
			& \leq \cPiLstabA\, \cembA 
			\norm{\abs{\e}^{1/2} \nabla \e}_{L^2}
			\norm{\e}_{L^3}^{1/2} \norm{\e}_{L^3}^{1/2}
			\left(\norm{\e}_{L^3}^{3} 
			+ \tfrac{9}{4}\norm{\abs{\e}^{1/2} \nabla\e}_{L^2}^2\right)^{1/2}
			\\
				& \leq  
			\delta \norm{\abs{\e}^{1/2} \nabla \e}_{L^2}^2 
			+ \frac{(\cPiLstabA\, \cembA)^2 }{4 \delta}
			\norm{\e}_{L^3}^{2}
			\left(\norm{\e}_{L^3}^{3} 
			+ \tfrac{9}{4}\norm{\abs{\e}^{1/2} \nabla\e}_{L^2}^2\right). 
		\end{aligned}
	\end{equation}
	Applying similar arguments we estimate the second term in \eqref{est:stabL3-4} as 
	\begin{equation}\label{est:stabL3-6}
		\begin{aligned}
			(\mathrm{II}) 
			&\leq
			\norm{\hat \bu}_{L^{6}} 
			\norm{\Pi(\abs{\e}\e)}_{L^{3}}
			\norm{\nabla \e}_{L^2}
			\leq
			\cembA \norm{\hat \bu}_{W^{1,2}} 
			\norm{\Pi(\abs{\e}\e)}_{L^{3}}
			\norm{\nabla \e}_{L^2}
			\\
			& \leq \cPiLstabA\, \cembA^2
			\norm{\hat \bu}_{W^{1,2}} 
			\norm{\e}_{L^3}^{1/2}
			\left(\norm{\e}_{L^3}^{3} 
			+ \tfrac{9}{4} \norm{\abs{\e}^{1/2} \nabla\e}_{L^2}^2\right)^{1/2}\norm{\nabla \e}_{L^2}\\
			& \leq
			 \varepsilon\norm{\nabla \e}_{L^2}^2 + 
			\frac{\cPiLstabA^2 \, \cembA^4 }{4 \varepsilon} 
			\norm{\hat \bu}_{W^{1,2}}^2  
			\norm{\e}_{L^3}
			\left(\norm{\e}_{L^3}^{3} + \tfrac{9}{4}\norm{\abs{\e}^{1/2}\nabla\e}_{L^2}^2\right).
		\end{aligned}
		\end{equation}
	Finally, on the last term of~\eqref{est:stabL3-4}, again with Hölder's inequality, estimate~\eqref{est:P3}, and Young's inequality we obtain
	\begin{equation}\label{est:stabL3-7}
	\begin{aligned}
		(\mathrm{III}) 
		&\leq
		\norm{\nabla \hat \bu}_{L^2} 
		\norm{\Pi(\abs{\e}\e)}_{L^3}
		\norm{ \e}_{L^6}
		\leq
		\norm{ \hat \bu}_{W^{1,2}} 
		\norm{\Pi(\abs{\e}\e)}_{L^3}
		\norm{ \e}_{L^6}
		\\
		& \leq \cPiLstabA  \,\cembA	  
		\norm{ \hat \bu}_{W^{1,2}} 
		\norm{\e}_{L^3}^{1/2}
		\left(\norm{\e}_{L^3}^{3} + \tfrac{9}{4}\norm{\abs{\e}^{1/2}\nabla\e}_{L^2}^2\right)^{1/2}\norm{ \e}_{L^6} \\
		& \leq  
		\delta \norm{\abs{\e}^{1/2}\nabla\e}_{L^2}^2 + \frac{4\delta}{9}\norm{\e}_{L^3}^3
		+\frac{9}{16}\frac{(\cPiLstabA \, \cembA)^2}{\delta } \norm{\e}_{L^3}	\norm{ \e}_{L^6}^2
		\norm{ \hat \bu}_{W^{1,2}}^2   \\
		&  \leq  
		\delta \norm{\abs{\e}^{1/2}\nabla\e}_{L^2}^2 
		+ \frac{4\delta}{9}\norm{\e}_{L^3}^3 + \frac{9}{16}\frac{\cPiLstabA^2 \, \cembA^4}{\delta } \norm{\e}_{L^3}	\norm{ \e}_{{W^{1,2}}}^2
		\norm{\hat \bu}_{W^{1,2}}^2. 
	\end{aligned}
\end{equation}
	Let us now estimate the residual term.  
	Employing duality, the estimate on $\Pi$ in Lemma~\ref{lem:Helholz-aux}, Young's inequality, and~\eqref{est:stabL3-10}, we can estimate 
	\begin{equation}\label{est:stabL3-11}
		\begin{aligned}
			 \skp{\bg}{\Pi(\abs{\e}\e)} \dx
			& \leq  
			\norm{\bg}_{W^{-1,3}} \norm{\Pi(\abs{\e}\e)}_{W^{1,3/2}}  \\
			&  \leq 
			\cPiLstabC\norm{\bg}_{W^{-1,3}}\left(
			\norm{\e}_{L^{3}}^2
			+ 2 \norm{\e}_{L^3}^{1/2}\norm{\abs{\e}^{1/2} \nabla\e}_{L^{2}}\right)  \\
			& \leq \delta\norm{\abs{\e}^{1/2} \nabla\e}_{L^{2}}^2 + \frac{\cPiLstabC^3}{3}\norm{\bg}_{W^{-1,3}}^3 + \frac{2}{3}\norm{\e}_{L^3}^{3} +
			\frac{\cPiLstabC^2}{\delta}\norm{\bg}_{W^{-1,3}}^2
			\norm{\e}_{L^3}  \\
			& \leq  \delta\norm{\abs{\e}^{1/2} \nabla\e}_{L^{2}}^2 +  \frac{\cPiLstabC^3}{3}
			\left(1 + \frac{2}{\delta^{3/2}}\right)
			\norm{\bg}_{W^{-1,3}}^3 + \norm{\e}_{L^3}^{3}.
		\end{aligned}
	\end{equation}
	Inserting \eqref{est:stabL3-2}--\eqref{est:stabL3-4} and \eqref{est:stabL3-11} with \eqref{est:stabL3-5}--\eqref{est:stabL3-7} into \eqref{est:stabL3-1} we obtain
	\begin{equation}
		\begin{aligned} \label{est:stabL3-12b}
			\frac{1}{3} \frac{\mathrm{d}}{\mathrm{d}t} \norm{\e}_{L^3}^3
			&+ \nu \bignorm{\abs{\e}^{1/2} \nabla \e}_{L^2}^2
			+ \frac{4 \nu}{9} \norm{\nabla (\abs{\e}^{3/2})}_{L^2}^2 
			\leq 
			\frac{\cPiLstabC^3}{3}  
			\left(1 + \frac{2}{\delta^{3/2}}\right)\norm{\bg}_{W^{-1,3}}^3 \\
			&  + 3\delta \norm{\abs{\e}^{1/2} \nabla\e}_{L^{2}}^{2}
			+ \varepsilon\norm{\nabla \e}_{L^2}^2\\
			& +
			\norm{\e}_{L^3}^{3} 
			\left(1 + \frac{4\delta}{9}
			\right)
			+ \norm{\e}_{L^3}^{3} (\cPiLstabA \, \cembA)^2
			\left(\frac{1}{4\delta } \norm{\e}_{L^3}^2 
			+ \frac{\cembA^2}{4\varepsilon } \norm{\e}_{L^3}\norm{\hat \bu}_{W^{1,2}}^2
			\right)\\
			& + 
			\tfrac{9}{4} (\cPiLstabA\, \cembA)^2\norm{\abs{\e}^{1/2}\nabla \e}_{L^2}^2 \left( 
			\frac{1 }{4 \delta}
			\norm{\e}_{L^3}^2 
			+ \frac{ \cembA^2 }{ 4\varepsilon} \norm{\e}_{L^3}\norm{ \hat \bu}_{W^{1,2}}^2
			\right) \\
			& + \tfrac{9}{4}(\norm{\e}_{L^2}^2 + \norm{\nabla\e}_{L^2}^2)\frac{\cPiLstabA^2\, \cembA^4}{ 4\delta} \norm{\e}_{L^3}\norm{\hat \bu}_{W^{1,2}}^2.
		\end{aligned}
	\end{equation}
	Adding the estimate in Lemma~\ref{lem:l2-stab} and choosing $\delta =\varepsilon = \frac{\nu}{4} $ shows that 
		\begin{equation}\label{est:stabL3-13}
		\begin{aligned}
			\frac{1}{3} \frac{\mathrm{d}}{\mathrm{d}t} \norm{\e}_{L^3}^3
			&
			+ \frac{1}{2} \frac{\mathrm{d}}{\mathrm{d}t}\norm{\e}_{L^2}^2 
			+ \frac{\nu}{4} \bignorm{\abs{\e}^{1/2} \nabla \e}_{L^2}^2
			+ \frac{\nu}{4} \norm{\nabla \e}_{L^2}^2 
			+ \frac{4 \nu}{9} \norm{\nabla (\abs{\e}^{3/2})}_{L^2}^2 \\
			& \leq	\frac{\cPiLstabC^3}{3}   
			\left(1 + \frac{2^4}{\nu^{3/2}} \right) 
			\norm{\bg}_{W^{-1,3}}^3 + \left(\frac{1}{2} + \frac{1}{\nu}\right)  \norm{\bg}_{W^{-1,2}}^2\\
			& \;\;\; +
			\norm{\e}_{L^3}^{3} 
			\left(1 + \frac{\nu}{9}\right) 
			+\norm{\e}_{L^{2}}^2\left(\frac{1}{2} + \frac{2 \cembA}{ \nu}\norm{\hat\bu}_{L^6}^{2} + 2 \frac{ 3^3 \cembA^2}{\nu^3}\norm{\hat\bu}_{L^6}^4\right)\\
			& \;\;\;
			+ 	
			9\frac{(\cPiLstabA \, \cembA)^2}{4\nu}
			\left( \tfrac{4}{9}\norm{\e}_{L^3}^{3}  
			+\norm{\abs{\e}^{1/2}\nabla \e}_{L^2}^2 
			+ \norm{\e}_{L^2}^2 + \norm{\nabla\e}_{L^2}^2
			\right)\\
			&\qquad \qquad 
			 \left( \norm{\e}_{L^3}^2 
			+ \cembA^2 \norm{\hat \bu}_{W^{1,2}}^2 \norm{\e}_{L^3}
			\right).
		\end{aligned}
	\end{equation}
	Integration in $(0,t)$ shows the claim. 
\end{proof}

Note that the right-hand side of the estimate in Proposition~\ref{prop:main-stab} is bounded, for strong solutions $\hat \bu$ and $\bu$. 

\subsection{Conditional stability estimate}

Using the following generalised Gronwall Lemma allows us to deduce conditional error estimates from the estimates in Proposition~\ref{prop:main-stab}. 

\begin{lemma}[generalised Gronwall Lemma~{\cite[Lem.~3.4]{BrunkGiesselmannLukacovaMedvidova2025}}]\label{lem:ggG}
	For $T>0$ let $g_1\in C([0,T])$, let $g_2\in L^1(0,T)$ and let $\alpha\in L^\infty(0,T)$ be non-negative functions  and let $A,B_1, B_2, \beta_1, \beta_2 > 0$ be real numbers such that 
	\begin{multline*}
		g_1(t) + \int_0^t g_2(s) \ds \leq A + \int_0^t \alpha(s)g_1(s) \ds + \sum_{i=1}^2 B_i\sup_{\tilde s\in[0,T]} g_1(\tilde s)^{\beta_i}\int_0^t g_1(s) + g_2(s) \ds
	\end{multline*}
	holds for all $ t \in [0,T]$. 
	We denote $M\coloneqq \exp(\int_0^T \alpha(s) \ds)$. 
	Provided that  
	\[ 8 (1+T) \left( B_1  (8AM)^{\beta_1} + B_2  (8AM)^{\beta_2}  \right) \leq 1   \]
	holds, the following estimate is satisfied
	\begin{equation*}
		\sup_{t\in[0,T]} g_1(t) + \int_0^T g_2(s) \ds \leq 2AM.
	\end{equation*}
\end{lemma}

The following conditional error estimate is key in our main result on conditional existence of strong solutions. 

\begin{theorem}[Conditional stability estimate]\label{thm:condstab}
	Let $\nu>0$ be constant, let $T>0$ and let $\cPiLstabA,\cPiLstabC, \cembA$ be the constants in~\eqref{est:stab-LerayA}--\eqref{est:emb}. 
	Let $\bu$ and $ \hat \bu$ be strong solutions to~\eqref{eq:u} on $[0,T')$ for some $T' \in (0,T]$ for  data $(\bu_0, \bf)$ and $(\hat \bu_0, \hat \bf)$, respectively, satisfying Assumption \ref{ass:data}. 
	Assume that
	\begin{equation}
		\label{eq:agg}
		8 (1+T') \left( B_1  (8AM)^{\beta_1} + B_2  (8AM)^{\beta_2}  \right) \leq 1 , 
	\end{equation}
	for $\bg = \bf - \hat \bf$, and $\e = \bu - \hat{\bu}$, for $\beta_1= 2/3$, $ \beta_2 = 1/3 $ and 
		\begin{align*}
			A &\coloneqq \tfrac{1}{3}\norm{\e_0}_{L^3}^3  
			+  \tfrac{1}{2} \norm{\e_0}_{L^2}^2 
			+  \int_{0}^{T'} \left( \tfrac{\cPiLstabC^3}{3} \left(1 + \tfrac{2^4}{\nu^{3/2}}\right)\norm{\bg}_{W^{-1,3}}^3 + \left(\tfrac{1}{2} + \tfrac{1}{\nu}\right)  \norm{\bg}_{W^{-1,2}}^2 \right) 
			\ds
			\\
			M& \coloneqq \exp\left( \int_0^{T'} \alpha(s) \ds \right), \quad \text{ with }
			\alpha(t) \coloneqq 
			4 + \frac{\nu}{3} +
			\frac{4 \cembA}{ \nu}\norm{\hat\bu}_{L^6}^{2} + 4 \frac{ 3^3 \cembA^2}{\nu^3}\norm{\hat\bu}_{L^6}^4,\\ 
			\quad
			B_1 &\coloneqq \tfrac{3^{8/3}}{2 \nu} \cPiLstabA^2 \cembA^2 \left( 1 + \tfrac{2}{\nu}\right) , \quad \text{ and } \quad 
			B_2 \coloneqq \tfrac{3^{7/3}}{2 \nu}\cPiLstabA^2\cembA^4 \left( 1+ \tfrac{2}{\nu}\right) \sup_{t \in [0,T']}\norm{\hat\bu}_{W^{1,2}}^2. 
		\end{align*}
	Then the following error estimate holds 
	\begin{equation}\label{eq:condest}
		\begin{aligned}
			\sup_{t\in[0,T']} &\left( \tfrac{1}{2}\norm{\e}_{L^2}^2 + \tfrac{1}{3}\norm{\e}_{L^3}^3\right)+ \tfrac{\nu}{4}
			\int_{0}^{T'} \bignorm{\abs{\e}^{1/2} \nabla \e}_{L^2}^2
			+ \norm{\nabla \e}_{L^2}^2
			\ds \\
			&\leq 2M\Big( \tfrac{1}{3}\norm{\e_0}_{L^3}^3  
			+  \tfrac{1}{2} \norm{\e_0}_{L^2}^2 \big)\\
	& \qquad + 2 M \Big(			
		 \int_{0}^{T'} \left( \tfrac{\cPiLstabC^3}{3} \left(1 + \tfrac{2^4}{\nu^{3/2}}\right)\norm{\bg}_{W^{-1,3}}^3 + \left(\tfrac{1}{2} + \tfrac{1}{\nu}\right)  \norm{\bg}_{W^{-1,2}}^2 \right) 
			\ds \Big). 
		\end{aligned}
		\end{equation}
	\end{theorem}
	
	\begin{proof}
		We observe that all assumptions of Lemma \ref{prop:main-stab} are satisfied. 
		Hence, we have
			\begin{align*}
				\tfrac{1}{3}  \norm{\e(t)}_{L^3}^3
				&
				+ \tfrac{1}{2} \norm{\e(t)}_{L^2}^2 
				+ \tfrac{\nu}{4}
				\int_{0}^t \bignorm{\abs{\e}^{1/2} \nabla \e}_{L^2}^2
				+ \norm{\nabla \e}_{L^2}^2
				\ds
				\\
				\leq  &
				\tfrac{1}{3}\norm{\e_0}_{L^3}^3  
				+  \tfrac{1}{2} \norm{\e_0}_{L^2}^2 
				+  \int_{0}^t \tfrac{\cPiLstabC^3}{3}\left(1 + \tfrac{2^4}{\nu^{3/2}}\right)\norm{\bg}_{W^{-1,3}}^3 + \left(\tfrac{1}{2} + \tfrac{1}{\nu}\right)  \norm{\bg}_{W^{-1,2}}^2 \ds
				\\
				&
				+ \int_{0}^t \left(\tfrac{1}{2} + \tfrac{2 \cembA}{ \nu}\norm{\hat\bu}_{L^6}^{2} + 2 \tfrac{ 3^3 \cembA^2}{\nu^3}\norm{\hat\bu}_{L^6}^4\right)\norm{\e}_{L^{2}}^2 \ds
				+  (1+\tfrac{\nu}{9})\int_{0}^t 
				\norm{\e}_{L^3}^{3} 
				\ds
				\\
				& + 
				\tfrac{9}{4\nu } \cPiLstabA^2 \cembA^2 
				\int_{0}^t  
				\left(\norm{\e}_{L^3}^2 + \cembA^2 \norm{\e}_{L^3}   \norm{ \hat \bu}_{W^{1,2}}^2 \right) \cdot \\
				& \qquad \qquad \qquad \qquad  
				\left( \tfrac{4}{9} \norm{\e}_{L^3}^{3} 
				+  \norm{\e}_{L^2}^2
				+  \norm{\abs{\e}^{1/2}\nabla \e}_{L^2}^2
				+ \norm{\nabla\e}_{L^2}^2
				\right)
				\ds 
			\end{align*}
		Application of the generalised Gronwall lemma, see Lemma~\ref{lem:ggG}, with
		\begin{align*}
			g_1 \coloneqq\tfrac{1}{3}  \norm{\e(t)}_{L^3}^3+ \tfrac{1}{2} \norm{\e(t)}_{L^2}^2, \qquad g_2 \coloneqq\tfrac{\nu}{4}\left(\bignorm{\abs{\e}^{1/2} \nabla \e}_{L^2}^2
			+ \norm{\nabla \e}_{L^2}^2 
			\right),
		\end{align*}
		and the condition~\eqref{eq:agg} concludes the proof.
	\end{proof}

	\begin{remark}\label{rmk:tim-exp}
	Below in Section~\ref{sec:num-scheme} the conditional stability estimate  will be applied for a reconstruction $\hat{\bu}$ of a numerical approximation, and $\bg$ takes the role of its residual. 
	One may expect, that the residual and hence also $A$ converges to zero algebraically in the discretisation parameters, if the solution $\bu$ is sufficiently smooth. 
	If by a priori estimates one can ensure, that the norms on $\hat{\bu}$ in $\alpha$ and in $B_2$ are bounded uniformly in the discretisation parameters, then for sufficiently small discretisation parameters, the condition~\eqref{eq:agg} would be satisfied (verifiability). 
	However, with the exponential dependence of $M$ on $T$, even if the solution to the Navier--Stokes equations is smooth, one may have to choose discretisation parameters that depend exponentially on the time interval. 
	
	Such a relationship was present in all previous works except for~\cite{HeywoodRannacher1982}. 
	This is not surprising, since the worst case behaviour is typically captured by a Gronwall type arguments. 
	In \cite{HeywoodRannacher1982} this can be avoided by imposing certain additional stability assumptions on the numerical solutions. 
	This has the implication, that verifiability can be expected only for a certain subclass of strong solutions with suitable stability properties. 
	\end{remark}

\section{Proof of  Theorem~\ref{thm:main-1}}\label{sec:main-proof}

In this section, we present and prove our main result in Theorem~\ref{thm:main-1}. 
It is based on the conditional stability estimate Theorem~\ref{thm:condstab} and on the blow-up criterion in the critical space $L^3$ by \cite{IskauriazaSereginShverak2003,AlbrittonBarker2020}, see Proposition~\ref{prop:blowup}.

\begin{proof}
	
	Recall that $\bu$ is a Leray--Hopf weak solution to the Navier--Stokes equations to the data $(\bu_0, \bf)$ as in Assumption~\ref{ass:data}. 
	Furthermore, $\hat \bu \in C([0,T];W^{1,2}_{\div}(\dom))$ is a Leray--Hopf solution to the data $(\hat \bu_0,\bf + \mathbf{r})$, where $\mathbf{r} = \mathbf{r}[\hat \bu] \in L^2(0,T;L^2(\dom)^3) $ is its residual. 
We assume that \eqref{eq:agg} is satisfied.

	By Proposition~\ref{prop:NSsols-reg} using the regularity of the data $(\bu_0, \bf)$ there is a maximal $T^* \in (0,T]$ such that $ \bu$ is a strong solution on $[0,T^*)$.  
	Note that with $\mathbf{r} \in L^2(0,T;L^2(\dom)^3)$ and $\hat \bu_0 \in W^{1,2}_{\div}(\dom)$ the fact that $\hat \bu \in L^\infty(0,T;W^{1,2}_{\div}(\dom))$ implies that $\hat \bu$ is a strong solution on $[0,T)$ with data $(\hat \bu_0, \bf + \mathbf{r})$, and in particular on $[0,T^*)$, see Definition~\ref{defn:strong}. 
	
	For contradiction let us assume that the statement does not hold, i.e., $\bu$ is not a strong solution on $[0,T^*]$, and hence  $T^* < T'$. 
	Then, by Proposition~\ref{prop:blowup}, we have that 
	\begin{align}\label{thm:main-blowup}
	\limsup_{t \nearrow T^*}\norm{\bu(t)}_{L^3} = \infty, 
	\end{align}
	i.e., there exists an increasing sequence $(t_k)_{k \in \mathbb{N}}$ such that $t_k \to T^*$ and 
	\begin{align}\label{thm:main-blowu2}
		\lim_{k \to \infty}\norm{\bu(t_k)}_{L^3} = \infty.
	\end{align} 
	
	Since  both $\hat \bu$ and $ \bu$ are strong solutions on each interval $[0,t_k] \subset [0,T^*)$ Proposition~\ref{prop:main-stab} applies. 
	By monotonicity of all the quantities involved, condition~\eqref{eq:agg} holds on each $[0,t_k]$. 
	Thus, by Theorem~\ref{thm:condstab} we have that $\sup_{t \in [0,t_k]}\norm{\bu(t) - \hat \bu(t)}_{L^3}$ is bounded, uniformly in $k$. 
	Because $\hat \bu$ is a strong solution with $\hat \bu \in L^\infty(0,T;L^3(\dom)^3)$ this implies that 
	\begin{align}
		\norm{\bu(t)}_{L^3} \leq  	\norm{\bu - \hat \bu(t)}_{L^3} + 	\norm{\hat \bu(t)}_{L^3} \leq c,  
	\end{align}
	for any $t \in [0,T^*)$. 
	This is a contradiction to \eqref{thm:main-blowu2}, and hence proves the claim. 
\end{proof}

% ----------------------------------------

\section{Numerical scheme (proof of Theorem~\ref{thm:main-2})}\label{sec:num-scheme}

In this section we prove Theorem~\ref{thm:main-2} constructively by presenting one example of $\hat{\bu}$, such that one can compute or estimate all quantities in condition~\eqref{eq:agg}. 

We introduce the numerical approximation  $\bu_{\tau h}$ based on a general conforming mixed finite element scheme in space and an implicit Euler time stepping in Section~\ref{sec:scheme}. 
However, such solutions do not satisfy the assumption on $\hat{\bu}$ in Theorem~\ref{thm:main-1}. 
For this reason, we define a reconstruction $\hat \bu$ of $\bu_{\tau h}$ as the solution of a suitable Stokes problem, detailed in Section~\ref{sec:reconst}. This ensures that $\hat \bu \in C([0,T];W^{1,2}_{\div}(\dom))$ and that condition~\ref{itm:main-res} on the residual is satisfied.  
 Employing standard  error estimators available for the Stokes problem, which we recall in Section~\ref{sec:stokes}, we derive a posteriori error estimates for all terms occurring in condition~\ref{itm:main-cond}. 
Thus, we can verify condition~\ref{itm:main-cond} in Theorem~\ref{thm:main-1}  using only properties of the numerical solution $\bu_{\tau h}$, and thus exclude blow-up if it holds. 
Let us recall that the reconstruction is  a purely theoretical tool, which need not be computed. 

\subsection{The numerical scheme}\label{sec:scheme}
In this subsection, we introduce relevant notation and the numerical scheme which we employ.
\smallskip 

\textbf{Space discretisation:} 
Let $\mathcal{T}_h$ be a conforming and shape regular partition of $\dom$ into cells, i.e., tetrahedra or into quadrilaterals. For every cell $K\in \mathcal{T}_h$ and faces $e\in\partial K$ we introduce $h_K= \operatorname{diam}(K)$ and $h_e= \operatorname{diam}(e)$. 
Furthermore, we set  $h \coloneqq \max_{K \in \mathcal{T}_h} h_K>0$ and we use it also as index of a  partition $\mathcal{T}_h$ in a family of partitions. 
	
\begin{assumption}[conforming inf-sup stable pair]\label{as:fem}
	We assume that for each $h>0$ we have a mixed pair of finite element spaces $(V_h, Q_h)$ such that 
	\begin{enumerate}[label = (\roman*)]
		\item (conformity)  $(V_h, Q_h)\subset W^{1,2}(\dom)^3 \times L^2_{\sim}(\dom)$; 
		\item (approximation property) $V_h$ contains all continuous, piecewise (multi)-linear functions on $\mathcal{T}_h$;
		\item  (inf-sup stability) the family $(V_h,Q_h)$ is inf-sup stable, uniformly in $h$, as defined in~\cite[Ch.~8]{BoffiBrezziFortin2013}.
	\end{enumerate} 
\end{assumption}

Note that the velocity space is assumed to be $W^{1,2}$-conforming, i.e., it consists of continuous, piecewise smooth functions. 
Examples for such mixed finite elements are the (generalised) Taylor--Hood elements, the MINI-Element, the (conforming) Crouzeix--Raviart element, the Bernardi--Raugel element, the quadrilateral $Q_k - P_k$ element, see, e.g.,~\cite[Ch.~8]{BoffiBrezziFortin2013}, and the reduced Taylor--Hood element~\cite{DieningStornTscherpel2022}. 
More examples of higher order elements are available in the literature. 

When using mixed finite element methods the velocity is approximated in the space of discretely divergence-free velocity functions 
\begin{align}\label{def:Vhdiv}
	V_{h,\div} \coloneqq \{\bv_h \in V_h \colon \skp{\div \bv_h}{q_h} = 0 \; \text{ for all } q_h \in Q_h\}. 
\end{align}
Those functions are in general not exactly divergence-free. 
To retain skew-symmetry of the convective term for not exactly divergence-free functions, it is customary to use the following modification of the convective term, see~\cite{Temam1984},
\begin{align}\label{def:conv-discr}
	\widetilde b(\bu,\bv,\w)
	&\coloneqq \tfrac{1}{2}\left(  b(\bu,\bv,\w)  -  b(\bu,\w,\bv)\right) .
\end{align}
\smallskip 

\textbf{Time discretisation:} 
For the time-discretisation let $\tau>0$ and define the time grid points $t_i = i \tau$ and $I_i \coloneqq (t_{i-1},t_i)$, for $i \geq 1$. %We denote by 
Furthermore, let $\bf^i$ be approximations to $\bf(t_i)$, such as certain integral means, or if $\bf$ is continuous, point values in time. 

We consider an implicit Euler time stepping and use the notation 
\begin{align}\label{def:dt}
	\dtau \bv_h^i \coloneqq \frac{1}{\tau} (\bv_h^i - \bv_h^{i-1}). 
\end{align} 
Analogously, one may work with variable time step size. 

\subsubsection*{Numerical scheme}
Starting from $\bu_h^0 \in V_{h,\div}$ in the $i$th time step for given $\bu_h^{i-1} \in V_{h,\div}$ the pair of functions $(\bu_h^i,\pi_h^i) \in (V_h,Q_h)$ is determined by
\begin{equation}\label{eq:ns-weak-discr}
	\begin{aligned}
		 \skp{\dtau \bu_h^i}{\bv_h}  + \tilde b(\bu_h^i,\bu_h^i,\bv_h) 
		+ \nu  \skp{\nabla\bu_h^i}{\nabla\bv_h} 
		- \skp{ \pi_h^i}{ \div \bv_h} 
	&	= \skp{\bf^i}{\bv_h},
		\\
	\skp{q_h}{ \div \bu_h^i} 
		&= 0,
	\end{aligned}
\end{equation}
for all  $\bv_h \in V_h$ and all $q_h \in Q_h$. 

\begin{lemma}
	For any $h,\tau>0$ there exists at least one solution to \eqref{eq:ns-weak-discr}.
\end{lemma}

We will consider discrete initial data $\bu_h^0$ stemming from the discrete Stokes projection of $(\bu_0,0)$, see~\eqref{eq:disc_stokes_proj}.

\subsection{Stokes equations}\label{sec:stokes}
\subsubsection*{Stokes equations} For given constant $\nu>0$, a vector $\bm \in \setR^3$, a function $\bF \in L^2_\sim(\dom)^3$ we seek   $(\bu,\pi) \in  W^{1,2}(\dom)^3 \times L^2_\sim(\dom)$ such that
\begin{subequations}\label{eq:stokes}
\begin{alignat}{3}
	\nu \skp{\nabla \bu}{\nabla \bv} - \skp{\pi}{\div \bv} &= \skp{\bF}{\bv} \qquad &&\text{for any } \bv \in W^{1,2}(\dom)^3,\\
	\skp{\div\bu}{q} &= 0 \quad &&\text{for any } q \in L^2_{\sim}(\dom), \\
	\la \u_k, 1\ra &= \bm_k\quad && k\in\{1,2,3\}.
\end{alignat}
\end{subequations}

\begin{lemma}[regularity {\cite[Thm.~2]{KelloggOsborn1976}}]\label{lem:reg_stokes}
	For given $\bF \in L^2_\sim(\dom)^3$ there exists a unique solution $(\bu,\pi)\in W^{1,2}(\dom)^3 \times L^2_\sim(\dom)$ to the Stokes problem~\eqref{eq:stokes} and it satisfies that 
	\begin{align}
		\bu \in W^{2,2}(\dom)^3\qquad \text{ and } \quad \pi \in W^{1,2}(\dom). 
	\end{align}
\end{lemma}
	
For given $(\bu,p)\in W^{1,2}_{\div}(\dom)\times L^2_{\sim}(\dom)$ we define the discrete Stokes projection $(\Pi_1\bu,\Pi_1 p)\in V_h\times Q_h$ as solution to	
\begin{subequations}\label{eq:disc_stokes_proj}
	\begin{alignat}{3}
	\la \nabla \Pi_1\bu,\nabla \bv_h \ra - \la \Pi_1 \pi, \div(\bv_h)\ra &= \la \nabla \bu,\nabla \bv_h \ra - \la  \pi, \div(\bv_h)\ra, \\
	\la \div(\Pi_1\bu),q_h \ra &= 0,\\ \qquad \la \Pi_1\bu_k, 1\ra &= \la \bu_k, 1\ra\quad \quad  k\in\{1,2,3\},
	\end{alignat}
	for all $(\bv_h,q_h)\in V_h\times Q_h$.
\end{subequations}

We assume that there are (quasi-)interpolation operators $I_V \colon W^{2,2}(\dom)^3 \rightarrow V_h$ and  $I_Q \colon W^{1,2}(\dom) \rightarrow Q_h$ and constants $c_{i1}, c_{i2}, c_{i3}$ such that for each cell $K$ and face $e$ we have
\begin{align}\label{eq:interpol}
 \| \bv - I_V \bv \|_{L^2(K)} &\leq  c_{i1} h_K^2 \| \bv \|_{W^{2,2}(K)},\\
 \| \bv - I_V \bv \|_{L^2(e)} &\leq c_{i2} h_K^{3/2} \| \bv \|_{W^{2,2}(K)},\\
 \| q - I_Q q \|_{L^2(K)} &\leq c_{i3} h_K \| v \|_{W^{1,2}(\omega_K)}, \label{eq:interpol-3}
\end{align}
for any $\bv \in W^{2,2}(\dom)^3$ and $q \in W^{1,2}(\dom)$,
where $\omega_K$ is the patch of cells sharing an edge with the cell $K$.
Explicit values for the constants appearing in equation \eqref{eq:interpol} for standard interpolation operators and Clément interpolation can be found in \cite{Verfuerth1999}.

\begin{theorem}[reliability {\cite[Thm.~4.70]{Verfurth2013}}, {\cite[Lem.~5.4]{BKM18}}]\label{thm:stokes}
	Let $(V_h,Q_h)$ be a pair of finite element spaces as in Assumption~\ref{as:fem}. 
	For given $\bF \in L^2_{\sim}(\dom)^3$, constant $\nu>0$ and $\bm \in \mathbb{R}^3$ let $(\bu, \pi) \in W^{1,2}_{\div}(\dom) \times L^2_\sim(\dom)$ be the solution to the Stokes equation~\eqref{eq:stokes}.  
	Let $(\bu_h,\pi_h) \in V_h \times Q_h$ be the corresponding finite element solution.  
	Then, one has
	\begin{align}\label{est:ap-L2}
		 \| \bu - \bu_h\|_{L^2(\dom)}
	\leq \tilde c \left(\sum_K \bar \mu_K^2 + h_K^4 \| \mathbf{F} - \mathbf{F}_K\|_{L^2(K)}^2
	\right)^{\frac12}\eqqcolon H_0(\bu_h, \pi_h, \mathbf{F}),
	\end{align}
	where $\mathbf{F}_K$ denotes the mean value of $\mathbf{F}$ on $K$ and 
	\begin{align*}
	\bar \mu_K^2 
&	\coloneqq 
	 h_K^4 \| \mathbf{F} +\nu \Delta \bu_h - \nabla \pi_h \|_{L^2(K)}^2 + h_K^2 \|\div \bu_h \|_{L^2(K)}^2\\	 
& \qquad + \sum_{ e \in \partial K} h_e^3 \| \jump{\mathbf{n}_e \cdot (\nabla \bu_h - \pi_h I)}\|_{L^2(e)}^2,
	\end{align*} 
	for each  $K \in \mathcal{T}$. 
	Here
  $\tilde c= c_{\mathrm{ell}}\max \{ c_{i1}, c_{i2},k c_{i3} \}$ 
	depends on the constants in \eqref{eq:interpol}--\eqref{eq:interpol-3}, on the number  $k$ of edges per element and on the elliptic regularity constant $c_{\mathrm{ell}}$ for the Stokes problem on $\dom$.
	Moreover, let
	\begin{align*}
	\eta_K^2 
	& \coloneqq 
	h_K^2 \| \mathbf{F} + \nu \Delta \bu_h - \nabla \pi_h \|_{L^2(K)}^2 + \|\div \bu_h \|_{L^2(K)}^2 \\
	& \qquad 
	+ \sum_{ e \in \partial K} h_e \| \jump{\mathbf{n}_e \cdot (\nabla \bu_h - \pi_h I)}\|_{L^2(e)}^2.
	\end{align*}
	Then there exists a constant $	c>0$
	(the product of the maximum of the error constants $C_{A,2,2}$ and $C_{A,4,2}$ of the Clément interpolant in \cite[Prop. 3.33]{Verfurth2013} and the maximal number of neighbours for any cell)	
	such that the error satisfies
	\begin{align}\label{est:ap-H1}
		\left( \| \nabla \bu - \nabla \bu_h\|_{L^2(\dom)}^2
	+ \| \pi - \pi_h\|_{L^2(\dom)}^2\right)^{\frac12}
	\leq c \left(\sum_K  \eta_K^2 
	\right)^{\frac12}.
	\end{align}
%	We set
%	\[ H_1(\bu_h, \pi_h, \mathbf{F}) \coloneqq
%	c \left(\sum_K  \eta_K^2 
%	\right)^{\frac12} +  H_0(\bu_h, \pi_h, \mathbf{F})
%	\]
	To bound $\norm{\bu - \bu_h}_{W^{1,2}(\dom)}$ above we set
	\[ H_1(\bu_h, \pi_h, \mathbf{F}) \coloneqq
	\left(
	c^2 \sum_K  \eta_K^2  +  (H_0(\bu_h, \pi_h, \mathbf{F}))^2 \right)^{1/2}.
	\]
\end{theorem}
\begin{proof}
	The proof of the second estimate~\eqref{est:ap-H1} can be found in~\cite[Thm.~4.70]{Verfurth2013}. 
	The $L^2$-estimate~\eqref{est:ap-L2} proceeds analogously and is presented in Appendix~\ref{app:stokes} for the convenience of the reader.  
\end{proof}

In accordance with Remark 2.1 in \cite{BKM18} we note that the error estimates do not rely on the discrete inf-sup condition.

\subsection{Reconstruction}\label{sec:reconst}
Let us now define suitable reconstructions of numerical solutions for which we can apply the stability theory in Section~\ref{sec:stab}. 
This follows the idea of \emph{Stokes reconstruction} from~\cite{KarakatsaniMakridakis2007}. 
Note that the reconstruction is the exact solution to a Stokes problem. 
In particular, we will not compute it exactly, but use a posteriori error estimates to control its deviation from the numerical solution.
\smallskip 

\textbf{Space reconstruction:}
Suppose we have a sequence of numerical solutions $(\bu_h^i, \pi_h^i) \in V_h \times Q_h$ to \eqref{eq:ns-weak-discr}, where $(\bu_h^0,\pi_h^0)=(\Pi_1\u_0,\Pi_1 0)$. 
The key is that they can be interpreted as Galerkin solutions to an auxiliary Stokes problem in a function space setting for which a posteriori error analysis is available, as summarised in Section~\ref{sec:stokes}. 
We define the reconstruction $(\hat \bu^i, \hat \pi^i) \in W^{1,2}_{\div}(\mathbb{T}^3) \times L^2_{\sim}(\mathbb{T}^3) $ at time $t_i$, $i \geq 0$ as weak solution to the Stokes problem  
\begin{equation}\label{def:stokesrec}
	\begin{split}
		\nu \Delta \hat \bu^i - \nabla \hat \pi^i &= \nu \Delta_h  \bu_h^i - \nabla_h \pi_h^i
		\\
		\div \hat\bu^i &=0, \qquad \la \hat\bu_k^i-\bu_{h,k}^i,1\ra=0, \quad k\in\{1,2,3\}.
	\end{split}
\end{equation}
 where the discrete Laplacian  $\Delta_h \bu_h$ and  the discrete gradient $\nabla_h \pi_h$ are defined as the unique elements of $V_h$ satisfying
\begin{alignat}{3}
 \langle \Delta_h \bu_h, \bv_h \rangle &= -  \langle \nabla \bu_h, \nabla \bv_h \rangle \qquad &&\text{ for all }  \bv_h \in V_h,\\
 \langle \nabla_h \pi_h, \bv_h \rangle& = -  \langle \pi_h, \div \bv_h \rangle \qquad &&\text{ for all } \bv_h \in V_h.
\end{alignat}
One crucial observation is that for $i \geq 0$  $(\bu_h^i, \pi_h^i)$ is the Galerkin solution to \eqref{def:stokesrec}, i.e., we can use the a posteriori error estimates for Stokes to control $\bu_h^i - \hat \bu^i$. For $i\geq 1$ the right-hand side of \eqref{def:stokesrec} can be rewritten as
\begin{align}\label{def:Fi}
	\nu \Delta_h  \bu_h^i - \nabla_h \pi_h^i = \bF^i(\bu_h) \coloneqq   \frac12 \bu_h^i \cdot \nabla \bu_h^i +  \frac12\div(\bu_h^i \otimes \bu_h^i) +
	\dtau \bu_h^i - \bf^i.
\end{align}

Note that $\mathbf{F}^i(\bu_h) \in L^2(\dom)^3$ holds,  provided that $V_h \subset W^{1,\infty}$, which is true for conforming mixed finite element spaces. 
The numerical scheme~\eqref{eq:ns-weak-discr} can be expressed as
\[
\nu \langle \nabla \bu_h^i, \nabla \bv_h \rangle 
-\langle \pi_h^i, \div \bv_h \rangle = -\langle \mathbf{F}^i(\bu_h), \bv_h\rangle, \qquad \langle \div(\bu_h^i),q_h\rangle = 0.
\]
for any $(\bv_h,q_h) \in V_h \times Q_h$. 
Since the left-hand side is zero for constant $\bv_h$ it follows that $\mathbf{F}^i(\bu_h)$ is mean-free.

Application of Lemma \ref{lem:reg_stokes} on elliptic regularity  implies that $(\hat\bu^i,\hat\pi^i)$ is in $ W^{2,2}(\dom)^3\times W^{1,2}_{\sim}(\dom)$ for each $i$.
\smallskip 

\textbf{Time reconstruction:} 
Starting from the space reconstruction, using affine interpolation in time on  $(\hat \bu^i, \hat \pi^i)_i$, we define space-time reconstructions $(\hat \bu, \hat \pi ) \in C([0,T]; W_{\div}^{1,2}(\dom)) \times C([0,T];L^2_{\sim}(\mathbb{T}^3))$. 
Denoting by $(\ell_i)_i$ the affine Lagrange basis with respect to the time grid $(t_i)_i$, i.e., $\ell_i(t_j) = \delta_{i,j}$ this yields on $I_i = (t_{i-1},t_i)$
\begin{align}\label{def:rec-time}
	\hat \bu |_{I_i} \coloneqq  \ell_{i-1} \hat \bu^{i-1}  + \ell_{i} \hat \bu^{i} \qquad \text{ and } \qquad 
	\hat \pi |_{I_i} \coloneqq  \ell_{i-1} \hat \pi^{i-1}  + \ell_{i} \hat \pi^{i}. 
\end{align} 
Using~\eqref{def:stokesrec} we find
\begin{align*}
	(\nu\Delta\hat\bu - \nabla \hat\pi)|_{I_i} =  \ell_{i-1} \bF^{i-1}(\bu_h)  + \ell_{i} \bF^i( \bu_h). 
\end{align*}
The reconstruction $\hat \bu$ is divergence-free, and we may define its residual $\mathbf{r}  = \mathbf{r}[\hat \bu]\in L^2(0,T;L^2(\dom)^3)$. 
Indeed, $(\hat \bu,\hat\pi)$ satisfies  
\begin{equation}
	\begin{split}
		\partial_t \hat \bu + \hat\bu \cdot \nabla \hat\bu + \nabla \hat \pi - \nu \Delta \hat \bu
		&= \mathbf{f} + \mathbf{r},\\
		\div \hat \bu &=0,
	\end{split}
\end{equation}
with residual 
\begin{align}\label{def:residual}
 \mathbf{r}^i \coloneqq	\mathbf{r}|_{I_i} \coloneqq 
	\dtau \hat \bu^i + \hat \bu \cdot  \nabla \hat \bu -
	\ell_{i-1} \bF^{i-1}(\bu_h) 
	 -	\ell_{i} \bF^i(\bu_h) - \bf. 
\end{align}

\begin{lemma}\label{lem:rec_prop12}
	The reconstruction $(\hat\bu,\hat \pi)$ in \eqref{def:stokesrec}, \eqref{def:rec-time} and the residual $\mathbf{r}$ in \eqref{def:residual} satisfy that \begin{align*}
		(\hat\bu,\hat \pi)\in C([0,T];W^{1,2}_{\div}(\dom))\times C([0,T];L^2_{\sim}(\mathbb{T}^3)) \quad \text{ and  } \quad \mathbf{r}\in L^2(0,T;L^2(\dom)^3).
	\end{align*} 
	Hence, $(\hat\bu,\hat \pi)$ is a strong solution with data $(\hat \bu(0), \mathbf{f} + \mathbf{r})$.
\end{lemma}

In the following we shall use several times that by~\eqref{est:emb} we  have 
\begin{align}\label{est:emb-a}
	W^{1,2}(\dom)^3 \hookrightarrow L^6(\dom)^3 \hookrightarrow L^3(\dom)^3
\end{align}
with embedding constant $\cembA$, since $\dom$ has measure $1$. 
Analogously, by~\eqref{est:emb2} we have 
\begin{align}\label{est:emb-b}
	W^{1,3/2}(\dom)^3 \hookrightarrow L^3(\dom)^3 \hookrightarrow L^2(\dom)^3
\end{align}
with embedding constant $\cembB$. 
The dual embedding has the same embedding constant
\begin{align}\label{est:emb-c}
	L^2(\dom)^3 \hookrightarrow L^{3/2}(\dom)^3 \hookrightarrow W^{-1,3}(\dom)^3. 
\end{align}

\subsection{Verification of condition \ref{itm:main-cond} in Theorem \ref{thm:main-1}}

In this subsection, we show how to verify condition~\ref{itm:main-cond} in Theorem \ref{thm:main-1}, i.e., we check whether~\eqref{eq:agg} holds.
Looking at \eqref{eq:agg} one needs to compute or estimate the norms 
\begin{align*}
	\norm{\e_0}_{L^2}, \;\;
	\norm{\e_0}_{L^3}, \;\;  
	\norm{\hat\bu^i}_{L^6},\;\; \norm{\hat\bu^i}_{W^{1,2}},\;\;
	\norm{\mathbf{r}^i}_{W^{-1,2}},\; \text{ and }  \;
	\norm{\mathbf{r}^i}_{W^{-1,3}}.
\end{align*}
For the error $\e_0 = \bu_0 - \hat \bu^0$ at initial time we apply Theorem~\ref{thm:stokes} and~\eqref{est:emb-a} to estimate
\begin{align}
\norm{\e_0}_{L^2} &\leq \norm{\bu_0-\bu_h^0}_{L^2} + \norm{\bu_h^0-\hat\bu^0}_{L^2} \leq \norm{\bu_0-\bu_h^0}_{L^2} + H_0[\bu_h^0,\pi_h^0,\bF^0(\bu_h)], \label{eq:est:eoL^2} \\
\norm{\e_0}_{L^3} &\leq \norm{\bu_0-\bu_h^0}_{L^3} + \norm{\bu_h^0-\hat\bu^0}_{L^3} 
\leq \norm{\bu_0-\bu_h^0}_{L^3} + \cembA H_1[\bu_h^0,\pi_h^0,\bF^0(\bu_h)]. \label{eq:est:eoL^3}
\end{align}
In similar fashion, we estimate 
\begin{align}
\norm{\hat\bu^i}_{L^6} &\leq \norm{\hat\bu^i-\bu^i_h}_{L^6} + \norm{\bu^i_h}_{L^6} \leq \cembA H_1[\bu_h^i,\pi_h^i,\bF^i(\bu_h)] + \norm{\bu^i_h}_{L^6}, \label{eq:est:hatbuL6}\\
\norm{\hat\bu^i}_{W^{1,2}} &\leq \norm{\hat\bu^i-\bu^i_h}_{W^{1,2}} + \norm{\bu^i_h}_{W^{1,2}} \leq H_1[\bu_h^i,\pi_h^i,\bF^i(\bu_h)] + \norm{\bu^i_h}_{W^{1,2}}. 
 \label{eq:est:hatbuH1}
\end{align}
Notably, all terms on the right-hand side of \eqref{eq:est:eoL^2}--\eqref{eq:est:hatbuH1} are computable. 
Finally we turn to estimates for the residual and note that it cannot be computed, since it contains the reconstruction. 
However, we can derive computable upper bounds for certain norms of the residual:

\begin{lemma}\label{lem:res}
 The residual defined in \eqref{def:residual} satisfies for $i \geq 2$
 \begin{align*}  \| \mathbf{r}^i\|_{L^\infty(I_i; W^{-1,3})} 
  \leq \;  & 
  \cembB H_0\left[\dtau\bu_h^i,\dtau \pi_h^i ,  \dtau\bF^i(\bu_h) \right]
  +
  \cembB \bignorm{ \tfrac{1}{\tau} (\bu_h^i - 2 \bu_h^{i-1} + \bu_h^{i-2}) }_{L^2}\\
  &
  +
  \cembB\cembA\nu^{-2}  c_P^2 \norm{ \bF^i(\bu_h) - \bF^{i-1}(\bu_h) }_{W^{-1,2}}^2\\
  &
  +
 \sum_{j\in \{i-1,i\}} \cembA(\cembB+\cembA)
		H_1 (\bu_h^j, \pi_h^j, \bF^j(\bu_h))\cdot  \\
&	\qquad\qquad \qquad  	(2 \|\bu_h^j\|_{W^{1,2}} + H_1 (\bu_h^j, \pi_h^j, \bF^j(\bu_h)))\\
 & + \cembB \| \ell_i\bf^i + \ell_{i-1}\bf^{i-1} - \bf \|_{L^\infty(I_i; L^2)}
 \end{align*}
where $c_P \coloneqq \frac{1}{2 \pi} \sqrt{4 \pi^2 + 1}$, and 
 \begin{align*}  \| \mathbf{r}^1\|_{L^\infty(I_1; W^{-1,3})} 
 \leq \;  &  
 \cembB H_0\left[\dtau\bu_h^1,\dtau\pi_h^1,\dtau\bF^1(\bu_h) \right]\\
  &
  +
  \cembB \Bignorm{ d_\tau \bu_h^1 + \tfrac{1}{2} \bu_h^0 \nabla \bu_h^0 + \operatorname{div}(\bu_h^0 \otimes \bu_h^0) - \nu \Delta_h \bu_h^0 + \nabla \pi_h^0 - \bf^0}_{L^2}\\
  &
  +
  \cembB\cembA\nu^{-2}  c_P^2 \norm{ \bF^1(\bu_h) - \bF^{0}(\bu_h) }_{W^{-1,2}}^2\\
  &
  +
  \sum_{j\in \{0,1\}} \cembA(\cembB+\cembA)
		H_1 (\bu_h^j, \pi_h^j, \bF^j(\bu_h))
	\cdot 
		\\
& \qquad \qquad \qquad
		(2 \|\bu_h^j\|_{W^{1,2}} + H_1 (\bu_h^j, \pi_h^j, \bF^j(\bu_h))) \\
 & + \cembB \| \ell_1\bf^1 + \ell_{0}\bf^{0} - \bf \|_{L^\infty(I_1; L^2)}. 
 \end{align*}
\end{lemma}
\begin{proof} 
We provide a detailed proof for $i \geq 2$. 
	For the sake of convenience we drop the index $i$ in the residual restricted to $I_i  = (t_{i-1},t_i)$. 
  We decompose the residual as
\[\r \coloneqq \r_1 + \r_2 + \r_3 + \r_4 + \r_5 + \r_6
\]
with
\begin{align}
	\r_1&\coloneqq \dtau\hat \bu^i - \dtau\bu_h^{i},\\
	\r_2 &\coloneqq \ell_i \left(\dtau\bu_h^i - \dtau\bu_h^{i-1} \right),\\
	\r_3 &\coloneqq - \ell_{i-1} \ell_i (\hat \bu^i - \hat \bu^{i-1}) \cdot \nabla (\hat \bu^i - \hat \bu^{i-1}),\\
	\r_4 &\coloneqq \frac{\ell_{i-1} }2  \left( \hat\bu^i \cdot \nabla \hat\bu^i -\bu_h^i \cdot \nabla \bu_h^i\right) +  
	\frac{\ell_i}2 \left( \hat\bu^{i-1} \cdot \nabla \hat\bu^{i-1} -\bu_h^{i-1} \cdot \nabla \bu_h^{i-1}\right),
	\\
	\r_5 &\coloneqq \frac{\ell_{i-1} }2 \left( \div(\hat\bu^i \otimes \hat\bu^i - \bu_h^i \otimes \bu_h^i)\right) +  
	\frac{\ell_i}2 \left(\div( \hat\bu^{i-1} \otimes \hat\bu^{i-1} - \bu_h^{i-1} \otimes \bu_h^{i-1})\right),\\
	\r_6 &\coloneqq \ell_i\bf^i + \ell_{i-1}\bf^{i-1} - \bf,
\end{align}
where we have used that $\hat \bu$ is divergence-free, and hence  
\[ \hat \bu \cdot  \nabla \hat \bu= \div(\hat \bu \otimes \hat \bu).\]
 Now, we discuss how to bound the contributions $\r_k$, $k = 1,\ldots, 6$, in the $W^{-1,3}$-norm.
 
 Concerning $\r_1$ we note that $\dtau\bu_h^i$ is a numerical approximation of the solution to a Stokes problem whose exact solution is $\dtau\hat \bu^i$. Thus $\r_1$
  can be estimated using error estimators for the Stokes problem since the Stokes problem is linear. Thus, applying \eqref{est:emb-c} and Theorem~\ref{thm:stokes} we have
  \[ \| \r_1\|_{L^\infty(I_i; W^{-1,3})} 
  \leq 
  \cembB \| \r_1\|_{L^\infty(I_i; L^2)}  
  \leq 
  \cembB H_0\left[\dtau\bu_h^i,\dtau \pi_h^i ,  \dtau\bF^i(\bu_h) \right].\]
 The term $\r_2$ can be computed directly since it only involves  discrete quantities 
 \[ \| \r_2\|_{L^\infty(I_i; W^{-1,3})} 
 \leq  \cembB
 \| \r_2\|_{L^\infty(I_i; L^2)}
  \leq \cembB 
  \bignorm{ \tfrac{1}{\tau}(\bu_h^i - 2 \bu_h^{i-1} + \bu_h^{i-2}) }_{L^2}.
 \]
Applying~\eqref{est:emb-c} and \eqref{est:emb-a} we can estimate $\r_3$ as follows
 \begin{equation}
\begin{aligned}
		\| (\hat \bu^i - \hat \bu^{i-1}) \cdot \nabla (\hat \bu^i - \hat \bu^{i-1})\|_{W^{-1,3}}
		& \leq 
		\cembB \| (\hat \bu^i - \hat \bu^{i-1}) \cdot \nabla (\hat \bu^i - \hat \bu^{i-1})\|_{L^{3/2}}
		\\
		& \leq 
		\cembB \| \hat \bu^i - \hat \bu^{i-1}\|_{L^6} \| \nabla (\hat \bu^i - \hat \bu^{i-1})\|_{L^{2}}\\
		&\leq 
		\cembB\cembA \| \hat \bu^i - \hat \bu^{i-1}\|_{W^{1,2}}^2.  
	\end{aligned}
	\end{equation}
Testing the  Stokes problems with the velocity implies the straightforward stability result
\begin{equation}\label{eq:stabstokes}
	\norm{ \hat \bu^i - \hat \bu^{i-1}}_{W^{1,2}} \leq  \nu^{-1} c_P \norm{ \mathbf{F}^i(\bu_h) - \mathbf{F}^{i-1}(\bu_h) }_{W^{-1,2}},
\end{equation}
where $c_P \coloneqq \frac{1}{2 \pi} \sqrt{4 \pi^2 + 1}$ appears due to the Poincaré inequality, which holds with constant $2\pi$ on $\dom$. 
Thus, it follows from the previous two estimates that 
\begin{equation}
 \|\r_3\|_{W^{-1,3}} \leq \cembB\cembA\nu^{-2} c_P^2 \norm{ \mathbf{F}^i(\bu_h) - \mathbf{F}^{i-1}(\bu_h) }_{W^{-1,2}}^2.
\end{equation}

In order to derive a computable bound for $\r_4$, we derive a bound for  $\|\hat\bu^i \cdot \nabla \hat\bu^i- \bu_h^i \cdot \nabla \bu_h^i\|_{W^{-1,3}}$. The other term in $\r_4$ can be estimated analogously. 
	Again, we use \eqref{est:emb-c} and Theorem~\ref{thm:stokes} to obtain
	\begin{align*}
	 \|\hat\bu^i \cdot \nabla \hat\bu^i-\bu_h^i \cdot \nabla \bu_h^i\|_{W^{-1,3}} &\leq 
		\cembB \|\hat\bu^i \cdot \nabla \hat\bu^i-\bu_h^i \cdot \nabla \bu_h^i\|_{L^{3/2}}\\		
		&\leq \cembB \| (\hat\bu^i- \bu_h^i) \cdot \nabla \bu_h^i\|_{L^{3/2}}
		+ \cembB \|\hat\bu^i \cdot \nabla (\hat\bu^i-\bu_h^i)\|_{L^{3/2}}\\
		&\leq
		\cembB \| \hat\bu^i- \bu_h^i\|_{L^6} \| \nabla \bu_h^i\|_{L^{2}}
		+ \cembB \| \hat\bu^i\|_{L^6} \| \nabla(\hat \bu^i- \bu_h^i)\|_{L^{2}}\\
		&\leq \cembB\cembA \| \hat\bu^i- \bu_h^i\|_{W^{1,2}} \| \nabla \bu_h^i\|_{L^{2}}
		\\
		& \qquad +\cembB \cembA (\| \bu_h^i\|_{W^{1,2}} + \| \hat \bu^i- \bu_h^i\|_{W^{1,2}}) \| \nabla(\hat \bu^i- \bu_h^i)\|_{L^{2}}\\
		&\leq 
		\cembB\cembA
		H_1 (\bu_h^i, \pi_h^i, \mathbf{F}^i(\bu_h))
		(2 \| \bu_h^i\|_{W^{1,2}} + H_1 (\bu_h^i, \pi_h^i, \mathbf{F}^i(\bu_h))).
	\end{align*}
	 We proceed  similarly for $\r_5$ using \eqref{est:emb-a} and obtain
	\begin{align*}
		\| \div (\hat \bu^i \otimes \hat \bu^i - \bu_h^i \otimes \bu_h^i)\|_{W^{-1,3}}
		&\leq \| \hat \bu^i \otimes \hat \bu^i - \bu_h^i \otimes \bu_h^i\|_{L^{3}}\\
		&\leq \cembA^2(\| \hat \bu^i\|_{W^{1,2}}+ \|\bu_h^i\|_{W^{1,2}}) \|  \hat \bu^i-\bu_h^i\|_{W^{1,2}}\\
		&\leq \cembA^2(2\|\bu_h^i\|_{W^{1,2}}+ \|  \hat \bu^i-\bu_h^i\|_{W^{1,2}}) \|  \hat \bu^i-\bu_h^i\|_{W^{1,2}}\\
		&
		\leq \cembA^2(2\|\bu_h^i\|_{W^{1,2}}+ H_1 (\bu_h^i, \pi_h^i, \mathbf{F}^i(\bu_h))) H_1 (\bu_h^i, \pi_h^i, \mathbf{F}^i(\bu_h)).
	\end{align*}
	For the final term $\r_6$ we estimate
\[ \| \r_6\|_{L^\infty(I_i; W^{-1,3})} \leq \cembB \| \r_6\|_{L^\infty(I_i; L^2)} \leq \cembB \| \ell_i\bf^i + \ell_{i-1}\bf^{i-1} - \bf \|_{L^\infty(I_i; L^2)}.
\]	

\end{proof}

\begin{remark}
	Note that the negative norm appearing in \eqref{eq:stabstokes} can be estimated as follows: Set $a_h \coloneqq  \mathbf{F}^i(\bu_h) - \mathbf{F}^{i-1}(\bu_h)$ and define $\phi \in W^{1,2}_{\sim}(\dom)^3$ as weak solution to  $-\Delta \phi=a_h$ subject to periodic boundary conditions and mean-free condition. 
	Then,  by the definition of the dual norm one has $ \| a_h\|_{W^{-1,2}}  \leq \| \nabla \phi \|_{L^2}$.
	Let $\phi_ h \in V_h$ be the corresponding finite element solution, i.e., satisfying 
	\[ 
	\skp{\nabla \phi_h}{ \nabla v_h} = \skp{ a_h }{v_h} %\quad \int _{\dom} \phi_h = \int_{\dom} \phi \quad
	\qquad \text{ for all } v_h \in V_h,
	\]
	subject to the mean-free condition, then we find
	\[  \| a_h\|_{W^{-1,2}}  \leq\| \nabla \phi_h \|_{L^2} + \| \nabla \phi- \nabla \phi_h \|_{L^2}\]
	where
	$\| \nabla \phi_h \|_{L^2}$ can be computed and  $\| \nabla \phi- \nabla \phi_h \|_{L^2}$ can be controlled using the corresponding Poisson error estimator. 
\end{remark}

The following Corollary shows Theorem~\ref{thm:main-2}.

\begin{corollary}[a posteriori criterion for existence of a strong solution]
Given a numerical solution $(\bu_h^i,\pi_h^i)$ to equation~\eqref{eq:ns-weak-discr}, condition~\eqref{eq:agg} can be verified a posteriori via the estimates \eqref{eq:est:eoL^2}--\eqref{eq:est:hatbuH1} and Lemma \ref{lem:res}. 
Consequently, this makes it practically possible to show that a solution $(\bu, \pi)$ exists up to a certain time, i.e., verifying the assumptions of Theorem \ref{thm:main-1}.
\end{corollary}
\begin{proof}
In Lemma \ref{lem:rec_prop12} we show that the reconstruction $(\hat\bu,\hat \pi)$ and the corresponding residual $\mathbf{r}$ satisfies the required regularity assumptions of Theorem~\ref{thm:main-1}.
Moreover, Lemma \ref{lem:res} provides a computable bound for the $W^{-1,3}$-norm of the residual, which dominates the $W^{-1,2}$-norm of the residual since on a cube with volume $1$ the continuous embedding $L^3 \hookrightarrow L^2$  has embedding constant $1$. Together with the estimates \eqref{eq:est:eoL^2}--\eqref{eq:est:hatbuH1} we are able to verify if condition \eqref{eq:agg} holds.

If the condition in Theorem~\ref{thm:condstab} is satisfied,  $\bu \in L^\infty(0,T';L^3(\dom)^3)$ holds automatically.
\end{proof}

\begin{corollary}[conditional a posteriori error estimate]\label{cor:apost}
Provided that the conditional stability estimate in~\eqref{eq:agg} holds, we obtain the following a posteriori error estimates
\begin{equation}\label{eq:fullaposteriori}
\begin{aligned} \sup_{t \in [0,T']} \Big( \tfrac{1}{2} \| \bu - \bu_{\tau h}\|_{L^2}^2
	&+ \tfrac{1}{3} \| \bu - \bu_{\tau h}\|_{L^3}^3 \Big) 
	+ \tfrac{\nu}{4} \int_0^{T'}  \| \nabla (\bu-\bu_{\tau h}) \|_{L^2}^2 \, \mathrm{d}s \leq 16AM \\
	&+ \sup_i H_0 \left[ \bu_h^i, \pi_h^i, \mathbf{F}^i(\bu_h)\right]^2 + 2\nu\sup_i H_1 \left[ \bu_h^i, \pi_h^i, \mathbf{F}^i(\bu_h)\right]^2 \\
	& + \tfrac{4}{3}\cembA^{3/2} \sup_i H_0 \left[ \bu_h^i, \pi_h^i, \mathbf{F}^i(\bu_h)\right]^{3/2} \sup_i H_1 \left[ \bu_h^i, \pi_h^i, \mathbf{F}^i(\bu_h)\right]^{3/2}
\end{aligned}
\end{equation}
where $\bu_{\tau h}$ denotes the piecewise linear in time interpolation of $(\bu_h^i)_i$.
\end{corollary}
\begin{proof}
	Let us look at these terms one by one on the time interval is $(0,T')$. 
	\[ \tfrac{1}{2} \| \bu - \bu_h\|_{L^\infty (L^2)}^2 \leq  \| \bu - \hat \bu_h\|_{L^\infty (L^2)}^2 + \| \hat \bu_h - \bu_h\|_{L^\infty (L^2)}^2 \]
	where we have a bound for the first term on the right-hand side by Theorem~\ref{thm:condstab}
	\begin{equation*}
		  \| \bu - \hat \bu_h\|_{L^\infty (L^2)}^2 \leq 4AM
	\end{equation*} 
	and the second term on the right-hand side is bounded by 
	\[
	\sup_i H_0 \left[ \bu_h^i, \pi_h^i, \mathbf{F}^i(\bu_h)\right]^2.
	\]
	In a similar way we can control the gradient contribution
	\begin{equation*}
		\nu\norm{\nabla (\bu-\bu_h)}_{L^2(L^2)}^2 \leq 2\nu\norm{\nabla (\bu-\hat\bu_h)}_{L^2(L^2)}^2 + 2\nu\norm{\nabla (\hat\bu_h-\bu_h)}_{L^2(L^2)}^2
	\end{equation*}
	 via $16AM$ and $2\nu\sup_i H_1 \left[ \bu_h^i, \pi_h^i, \mathbf{F}^i(\bu_h)\right]^2$ respectively.
	Concerning the second term in equation \eqref{eq:fullaposteriori} we have 
	\[ \tfrac{1}{3} \| \bu - \bu_h\|_{L^\infty (L^3)}^3 \leq \tfrac{4}{3}  \| \bu - \hat \bu_h\|_{L^\infty (L^3)}^3 + \tfrac{4}{3} \| \hat \bu_h - \bu_h\|_{L^\infty (L^3)}^3, \]
	where we have a bound for the first term on the right-hand side by Theorem~\ref{thm:condstab}
		\begin{equation*}
		\| \bu - \hat \bu_h\|_{L^\infty (L^3)}^3 \leq 6AM
	\end{equation*} 
	 and the second term satisfies
	\begin{align*}
		\| \hat \bu_h - \bu_h\|_{L^\infty (L^3)}^3 &\leq \cembA^{3/2}
		\| \hat \bu_h - \bu_h\|_{L^\infty (L^2)}^{3/2}\| \hat \bu_h - \bu_h\|_{L^\infty (W^{1,2})}^{3/2}\\
		&\leq \cembA^{3/2} \sup_i H_0 \left[ \bu_h^i, \pi_h^i, \mathbf{F}^i(\bu_h)\right]^{3/2} \sup_i H_1 \left[ \bu_h^i, \pi_h^i, \mathbf{F}^i(\bu_h)\right]^{3/2}.  
	\end{align*}
Since Theorem~\ref{thm:condstab} holds for the above norm simultaneously we can estimate those terms overall by $16AM.$	
	
\end{proof}

% -----------------------------------

\section{Conclusion and outlook}

We have presented an a posteriori existence result for strong solutions to the three-dimensional incompressible Navier--Stokes equations under periodic boundary conditions, based on the blow-up criterion in the critical space $L^\infty(L^3)$. 
Our approach leverages conditional stability estimates in $L^2$ and $L^3$, and yields a fully computable verification criterion expressed in terms of negative Sobolev norms of the residual. 
Importantly, the criterion applies directly to numerical solutions obtained by mixed finite element methods with implicit Euler time discretisation, without requiring additional assumptions on the underlying solution. 
Although the verification is limited to finite time intervals, the framework provides a rigorous pathway for extending existence verification to longer times, given sufficient computational resources. 
This work thus demonstrates the potential of combining analytical blow-up criteria with a posteriori error analysis to bridge the gap between numerical approximation and rigorous mathematical existence results for the Navier--Stokes equations.

Several directions for further investigation emerge from this study. 
One natural question concerns the comparison of the proposed verification result with available lower bounds for the existence time. 
Another avenue is the rigorous analysis, or experimental confirmation, of the convergence of residuals to zero as the discretisation parameters vanish, together with the determination of corresponding convergence rates. 
Furthermore, it is of interest to explore whether the approach can be extended to more general boundary conditions and domains beyond the torus.

\subsection*{Acknowledgements}
We acknowledge helpful discussions with Tobias Barker. 
The research of J.G. was supported by Deutsche Forschungsgemeinschaft (DFG, German Research Foundation) - SPP 2410 Hyperbolic Balance Laws in Fluid Mechanics: Complexity, Scales, Randomness (CoScaRa), within the Project “A posteriori error estimators for statistical solutions of barotropic Navier-Stokes equations” 525877563. 
The work of J.G. and T.T. is also supported by the Graduate School CE within Computational Engineering at Technische Universität Darmstadt. 
The work by T.T. was supported by  the German Research Foundation (DFG) via grant TRR 154, subproject C09, project number 239904186.
The work of A.B. was supported by the DFG via TRR 146, subproject C3, project number 233630050 and via SPP 2256 within project "Variational quantitative phase-field modeling and simulation of powder bed fusion additive manufacturing" project number 441153493.
\appendix

\section{Auxiliary results}
\label{sec:app-const}

In the following we collect explicit constants in the embedding and stability estimates in Section~\ref{sec:prelim}.  
Let us note that they are by no means optimal. 

\subsection{Embedding constants on periodic domains}

The following embedding results follow from \cite{MizuguchiTanakaSekineEtAl2017}, Sec.~4.3 for cubic domains, Thm.~3.1 and 3.2 therein. 
We have computed bounds on the constants with  Matlab. 
In particular, this yields explicit constants in the estimates in \eqref{est:emb} and \eqref{est:emb2}. 
Note however, that we did not succeed to reproduce the constants in ~\cite[Tab.~2 and 6]{MizuguchiTanakaSekineEtAl2017}, based on Thm.~3.1 therein. 

\begin{lemma}[embedding constants] \hfill
	\begin{enumerate}[label = (\roman*)]
	\item	One has
	\begin{align*}
		\norm{\bv}_{L^6(\dom)} \leq 24 \norm{\bv}_{W^{1,2}(\dom)} \qquad \text{ for any } \bv \in W^{1,2}(\dom)^3.	
		\end{align*}
	i.e., for the embedding $W^{1,2}(\dom)^3 \hookrightarrow L^6(\dom)^3$ the embedding constant is bounded by $\cembA= 24$, see \eqref{est:emb}.
	\item One has
	\begin{align*}
		\norm{\bv}_{L^3(\dom)} \leq 22 \norm{\bv}_{W^{1,3/2}(\dom)} \qquad \text{ for any } \bv \in W^{1,3/2}(\dom)^3.
	\end{align*}
	i.e., for the embedding $W^{1,3/2}(\dom)^3 \hookrightarrow L^3(\dom)^3$ the embedding constant is bounded by $\cembB= 22$. 
	%, see \eqref{est:emb2}.
	\end{enumerate}
\end{lemma}
\begin{proof}
The first estimate for scalar functions   follows from application of \cite[Thm~3.2]{MizuguchiTanakaSekineEtAl2017} with $p = 6$ with constant $21$; see also \cite[Tab.~6]{MizuguchiTanakaSekineEtAl2017}. 
Then one can show that the vector-valued version holds with constant $24$. 

The second estimate follows similarly from \cite[Thm.~2.1 \& 3.1]{MizuguchiTanakaSekineEtAl2017} with $p = 3$, $q = \tfrac{3}{2}$ using a decomposition of $\dom$ into $37$ cubes. 
For the scalar version one has constant $16$, and then the vector-valued version holds with constant $22$. 
\end{proof}

\subsection{Stability of the Helmholtz projection}\label{app:sec:stab}

In the following we collect explicit (but not optimal) stability constants of the Helmholtz projection in $L^p (\dom)$ and $W^{1,p}(\dom)$ for $p \in (1,\infty)$. 
From this we derive explicit constants in the stability estimates \eqref{est:stab-LerayA} and \eqref{est:stab-LerayB}.

\begin{lemma}\label{lem:Leray-stab}
For any $p \in (1,\infty)$ with $\tfrac{1}{p} + \tfrac{1}{p'} = 1$ we denote the constants 
\begin{align}\label{def:cp}
	c_p \coloneqq 2(p^\star-1), \quad \text{ for } \quad p^\star \coloneqq \max(p,p')
\end{align}
as well as 
\begin{align}\label{def:cp-01}
	c_{p}^{(0)} \coloneqq 1 + 
	\sqrt{3}\, c_p^2 
		\qquad 	\text{ and } \qquad c_{p}^{(1)} \coloneqq 1 + 3^{\frac{p^\star-1}{ p^\star}} \, c_p^2.
\end{align}
The Helmholtz projection $\Pi$ as defined in \eqref{eq:L2-proj} satisfies 
the following stability estimates:
\begin{alignat}{3}\label{est:Pip-1}
	 \norm{\Pi \bv}_{L^p(\dom)} &\leq c_p^{(0)} \norm{\bv}_{L^p(\dom)} \qquad  && \text{ for any } \bv \in L^p(\dom)^3,\\ \label{est:Pip-2}
	  \norm{\nabla \Pi \bv}_{L^p(\dom)} &\leq c_p^{(1)} \norm{\nabla \bv}_{L^p(\dom)} \quad &&\text{ for any } \bv \in W^{1,p}(\dom)^3.
\end{alignat}
In particular, $\Pi$ uniquely extends to any $L^p(\dom)^3$, for  $p \in (1,\infty)$. 
\end{lemma}
\begin{proof} 
	Note that by duality it suffices to prove the $L^p$-estimate in~\eqref{est:Pip-1} only for $p<2$, and then for $p>2$ the estimate holds with the constant $c_p^{(0)} = c_{p'}^{(0)}$. 
	 This  is reflected in the fact that $c_p^{(0)}$ only depends on $p^\star = \max(p,p')$. 
	
	Recall from \eqref{eq:L2proj-repr} that for $\bv \in L^2(\dom)^d$ we have 
	\begin{align*}
		\Pi \bv = \bv + \nabla (- \Delta)^{-1} \div \bv,  
	\end{align*}
	where $- \Delta \colon W^{1,2}_{\sim}(\dom) \to (W^{1,2}_{\sim}(\dom))'$, with $\div \bv \in (W^{1,2}_{\sim}(\dom))'$.
	
Thus, to prove~\eqref{est:Pip-1} and~\eqref{est:Pip-2}, respectively, it suffices to show that 
\begin{alignat}{3}\label{est:Pip-1-v}
	\norm{\nabla (- \Delta )^{-1} \div \bv}_{L^p(\dom)} &\leq \overline c_p^{(0)} \norm{\bv}_{L^p(\dom)} \qquad  && \text{ for any } \bv \in L^p(\dom)^3,\\ \label{est:Pip-2-v}
	\norm{\nabla^2 (- \Delta )^{-1} \div \bv}_{L^p(\dom)} &\leq \overline c_p^{(1)} \norm{\nabla \bv}_{L^p(\dom)} \quad &&\text{ for any } \bv \in W^{1,p}(\dom)^3.
\end{alignat}
for some constants $\overline c_p^{(0)}, \overline c_p^{(1)} >0$. 
Then, by triangle inequality \eqref{est:Pip-1} and \eqref{est:Pip-2} hold with $c_p^{(0)} = 1 + \overline c_p^{(0)} $ and $c_p^{(1)} = 1 + \overline c_p^{(1)} $, respectively. 
Furthermore, the latter estimate follows from 
\begin{alignat}{3}\label{est:Pip-2-vv}
	\norm{\nabla^2 (- \Delta )^{-1} g }_{L^p(\dom)} &\leq \widetilde c_p^{(1)} \norm{g}_{L^p(\dom)} \quad &&\text{ for any } g\in L^{p}(\dom),
\end{alignat}
for some constant $ \widetilde c_p^{(1)} >0$. 
Indeed, applying it to $g = \div \bv$ and noting that $\abs{\div \bv } \leq \sqrt{3} \abs{\nabla \bv}$ shows that \eqref{est:Pip-2-vv}  implies \eqref{est:Pip-2-v} with $\overline c_p^{(1)} = \sqrt{3} \widetilde c_p^{(1)} $. 
It remains to prove~\eqref{est:Pip-1-v} and~\eqref{est:Pip-2-vv}. 
We consider the operators 
\begin{align}
	T \bv = \nabla (- \Delta)^{-1} \div \bv \quad \text{ and } \qquad S g = \nabla^2  (- \Delta)^{-1} g,
\end{align}
for vector-valued $\bv$, and scalar-valued $g$. 
Note that $T\bv$ is vector-valued, whereas $Sg$ is matrix-valued. 
With this notation~\eqref{est:Pip-1-v} and \eqref{est:Pip-2-vv} are equivalent to
\begin{alignat}{3}\label{est:Pip-1-T}
	\norm{T \bv}_{L^p(\dom)} &\leq \overline c_p^{(0)} \norm{\bv}_{L^p(\dom)} \qquad  && \text{ for any } \bv \in L^p(\dom)^3,\\  \label{est:Pip-1-S}
	\norm{S g }_{L^p(\dom)} &\leq \widetilde c_p^{(1)} \norm{g}_{L^p(\dom)} \quad &&\text{ for any } g \in L^{p}(\dom). 
\end{alignat}
Both operators $T$ and $S$ are linear and their Fourier multipliers are matrix-valued, of the form 
\begin{align}
m(\xi) = - \frac{\xi \otimes \xi}{\abs{\xi}^2},
\end{align}
and homogeneous of degree $0$. 
Noting that the Riesz transform $R_j$ for $j \in \{1,2,3\}$ has the multiplier $m_j(\xi) = - i \frac{\xi_j}{\abs{\xi}}$ we can express $T$ and $S$ as 
\begin{alignat}{3}\label{eq:T-Riesz}
	(T \bv)_j &= - \mathcal{R}_j \sum_{k = 1}^3 \mathcal{R}_k \bv_k \qquad &&\text{ for } j \in \{ 1,2,3\},\\ \label{eq:S-Riesz}
	%\qquad \text{ for } j,k \in \{1,2,3\}.
	(S g)_{j,k}& = - \mathcal{R}_j \mathcal{R}_k g \qquad &&\text{ for } j,k \in\{ 1,2,3\}.
\end{alignat}
For the (vector-valued) Riesz transform $\mathcal{R}$ on $\setR^n$, defined by  $\mathcal{R} f = (\mathcal{R}_1 f, \mathcal{R}_2 f, \mathcal{R}_3 f)$,
 boundedness in $L^p$ is known with explicit constant. 
 The following estimate is  originally due to \cite{BanuelosWang1995}, see also \cite[Cor.~0.2]{DragicevicVolberg2006}: 
\begin{align}\label{est:Riesz}
\norm{\mathcal{R} f}_{L^p(\setR^n)} \leq c_p \norm{f}_{L^p(\setR^n)} \qquad \text{ for any } f \in L^p(\setR^n),
\end{align}
is presented with $c_p$ as above in~\eqref{def:cp} independent of $n \in \mathbb{N}$. 

Furthermore, by applying embedding results in $\ell_p$ and Hölder's inequality one can show that for vector-valued functions one has 
\begin{align}\label{est:Riesz-vec}
	\norm{\mathcal{R} f}_{L^p(\setR^n)} \leq  c_p\, m^{\frac{p^\star-2}{2p^\star}} \norm{f}_{L^p(\setR^n)} \qquad \text{ for any } f \in L^p(\setR^n)^m,
\end{align}
where $m \in \mathbb{N}$. 
By transference, see, e.g., \cite[Sec.~4.3.2]{Grafakos2014}, the analogous estimate holds for the Riesz transform on the torus $\dom$ with the same constant. 

Let us first prove the estimate~\eqref{est:Pip-1-T}, and recall that it suffices to consider $p<2$.    
Thus, for $T$ in \eqref{eq:T-Riesz}, applying \eqref{est:Riesz}, the triangle inequality, again \eqref{est:Riesz} to each of the terms in the sum and finally twice Hölder's inequality with $\tfrac{1}{p} + \tfrac{1}{p'} = 1$ and with $\tfrac{p}{2} + \tfrac{2-p}{2} = 1$, respectively, we obtain
\begin{align*}
\norm{T \bv}_{L^p(\dom)} 
	&\leq 
	c_p \Bignorm{\sum_{k = 1}^{3} \mathcal{R}_k\bv_k}_{L^p(\dom)} 
	 \leq 
		c_p \sum_{k = 1}^{3} \norm{\mathcal{R}_k\bv_k}_{L^p(\dom)} \\
	&\leq  c_p^{2} \sum_{k = 1}^3  \norm{ \bv_k}_{L^p(\dom)}
	\leq 
	 3^{\tfrac{p-1}{p} + \tfrac{2-p}{2p}}   c_p^2
 \norm{\bv}_{L^p(\dom)}
 = \sqrt{ 3}  c_p^2
 \norm{\bv}_{L^p(\dom)}. 
\end{align*}
Hence, for $p \leq 2$ the estimate~\eqref{est:Pip-1-T} holds with constant
 \begin{align*}
	\overline c_p^{(0)} \coloneqq  
	\sqrt{ 3}  c_p^2.
\end{align*} 
This  proves~\eqref{est:Pip-1-v}, and thus \eqref{est:Pip-1}, first for $p<2$, an then by duality also for $p \geq 2$. 

Now, let us prove the estimate~\eqref{est:Pip-1-S}.  
Applying for~\eqref{eq:S-Riesz} stability of the Riesz transform for vector-valued functions~\eqref{est:Riesz-vec}, and then the one for scalar functions~\eqref{est:Riesz} yields 
\begin{align*}
	\norm{S g}_{L^p(\dom)}
	&\leq 
3^{\frac{p^\star-2}{2 p^\star}}	c_p \norm{ \mathcal{R} g }_{L^p(\dom)} 
	\leq 
3^{\frac{p^\star-2}{2 p^\star}}	c_p^2 \norm{g }_{L^p(\dom)},
\end{align*}
i.e., \eqref{est:Pip-1-S} holds with $\widetilde c_p^{(1)} = 3^{\frac{p^\star-2}{2 p^\star}} c_p^2$, and so does \eqref{est:Pip-2-vv}, and finally \eqref{est:Pip-2-v} is satisfied with 
\begin{align*}
\overline c_p^{(1)} = 3^{\tfrac{1}{2}+ \frac{p^\star-2}{2 p^\star}} c_p^2 =  3^{ \frac{p^\star-1}{ p^\star}} c_p^2. 
\end{align*}
This finishes the proof of \eqref{est:Pip-1} and \eqref{est:Pip-2}. 
\end{proof}

On $\mathbb{T}^n$, $n \geq 2$ the corresponding estimates hold, with $3$ replaced by $n$ in~\eqref{def:cp-01}. 
One can improve the $L^p$-estimates by interpolation, which we shall do only in the following special cases. 

\begin{corollary}
The Helmholtz projection $\Pi$ defined in \eqref{eq:L2-proj} satisfies 
	\begin{alignat}{3}\label{est:Pip-1-3} 
		\norm{\Pi \bv}_{L^3} 
		&\leq 14  \norm{ \bv}_{L^3} \qquad &&\text{ for any } \bv \in L^{3}(\dom)^3,\\ \label{est:Pip-2-32}
		\norm{\Pi \bv}_{W^{1,3/2}} 
		&\leq 35 \norm{ \bv}_{W^{1,3/2}} \quad &&\text{ for any } \bv \in W^{1,3/2}(\dom)^{3},
	\end{alignat}
	i.e.,  \eqref{est:stab-LerayA} and \eqref{est:stab-LerayB} hold with $\cPiLstabA= 14$ and $\cPiLstabC = 35$, respectively.  
\end{corollary}
\begin{proof}
The stability estimates are a  consequence of Lemma~\ref{lem:Leray-stab} for specific $p$, interpolation and employing the definition of the Sobolev norm.

For  $p = 3$ and $p = 3/2$ we have $p^\star = 3$, and $c_p = 4$ as well as 
\begin{alignat*}{3}
		c_{p}^{(0)} &= 1 + 3^{1/2}\, 4^2 \leq 28.8 \leq 29 \qquad  &&\text{ for } p \in \{3/2,3\}, \\ 
		c_{p}^{(1)}  &= 1 + 3^{2/3}\, 4^2 \leq 34.3 \leq 35 &&\text{ for } p \in \{3/2,3\}. 
\end{alignat*}
This shows~\eqref{est:Pip-2-32}. 

Clearly, this would also give a bound in the $L^3$ estimate, but this can be improved. 
Note that by definition $\Pi$ is $L^2$-stable with constant $c_2^{(0)} = 1$. 
To obtain a stability constant in the $L^3$-estimate we interpolate between $L^2(\dom)$ and $L^6(\dom)$, noting that $\tfrac{1}{3} =  \tfrac{\theta}{2} +   \tfrac{1-\theta}{6}$ with $\theta = \tfrac{1}{2}$. 
For $p = 6$ we have $c_p = 10$, and hence  
\begin{align*}
	c_6^{(0)} = 1 + 3^{1/2}\, 10^2 \leq 174.3.
\end{align*}
 Then, by interpolation the stability in 
 \eqref{est:Pip-1-3} holds with constant 
 \begin{align*}
 (c_6^{(0)})^{1/2} \leq 13.2 \leq 14.	
 \end{align*}
\end{proof}

\subsection{$L^2$-error estimators for Stokes equations}\label{app:stokes}

Let us prove the $L^2$-estimate in Theorem~\ref{thm:stokes}, which follows similarly as the standard estimates in~\cite{Verfurth2013}.   
	We set $\e_u = \bu - \bu_h$ and $e_\pi= \pi - \pi_h$ and we define $\bz$ and $s$ as the solution to
	\[ 
	- \nu \Delta \bz - \nabla s = \e_u  \quad \text{ and } \quad   \operatorname{div} \bz =0,
	\]
	where $\bz$ and $s$ have mean value zero. We have elliptic regularity so that
	\begin{equation}\label{eq:ellregStokes}
		\| \bz \|_{W^{2,2}} + \| s \|_{W^{1,2}} \leq c_{\text{ell}} \| \e_u\|_{L^2}.
	\end{equation}

	Then we have for any $\bz_h \in V_h$, $s_h \in Q_h$ because of Galerkin orthogonality
	\begin{align*}
		\| \e_u\|_{L^2}^2 &= \langle \e_u, - \nu \Delta \bz - \nabla s\rangle \\
		&=\sum_K \int_K \nu \nabla \bz \cdot  \nabla \e_u 
		+ s \operatorname{div}(\e_u) \dx\\
		&= \sum_K \int_K \nu \nabla \bz \cdot \nabla \e_u 
		+ s \operatorname{div}(\e_u) - e_\pi \div \bz \dx\\
		&=  \sum_K \int_K \nu \nabla (\bz - \bz_h)  \cdot\nabla \e_u 
		+ (s-s_h) \operatorname{div}(\e_u) - e_\pi \div  (\bz - \bz_h)  \dx\\
		&= \sum_K \int_K (\mathbf{f} + \nu \Delta \bu_h - \nabla \pi_h) \cdot (\bz - \bz_h)
		+ (s-s_h) \operatorname{div}(\e_u)  \dx\\
		& \qquad 
		+ \sum_K \sum_{e \in \partial K} \int_e n_e (\nu \nabla \e_u - e_\pi I)\cdot (\bz - \bz_h) \,\mathrm{d} \sigma \\ 
		&\leq \left( \sum_K h_K^2 \| \operatorname{div} \e_u \|^2_{L^2}\right)^{\frac12} \left( \sum_K h_K^{-2} \|s - s_h\|^2_{L^2}\right)^{\frac12} \\
		& \qquad +
		\left( \sum_K h_K^4 \| \mathbf{f} + \nu \Delta \bu_h - \nabla \pi_h\|^2_{L^2}\right)^{\frac12} \left( \sum_K h_K^{-4} \| \bz - \bz_h \|^2_{L^2}\right)^{\frac12} \\
		& \qquad +
		\left( \sum_K  \sum_{e \in \partial K} h_e^3 \|n_e (\nu \nabla \e_u - e_\pi I) \|^2_{L^2(e)}\right)^{\frac12} \left( \sum_K h_e^{-3} \| \bz - \bz_h\|^2_{L^2(e)}\right)^{\frac12}
	\end{align*}
	Then, we can choose $\bz_h$ and $s_h$ as suitable interpolants into piecewise linear functions (so that we have explicit error bounds from \cite{Verfuerth1999}) and obtain
	\begin{equation}
	\begin{aligned}
		\| \e_u\|_{L^2}^2
		 \leq 
		\max\{ c_{i1},c_{i,2},k c_{i3}\}\Big( \sum_K h_K^2 \| \div \e_u \|^2_{L^2} + h_K^4 \| \mathbf{f} + \nu \Delta \bu_h - \nabla \pi_h\|^2_{L^2}\\
		+\sum_{e \in \partial K} h_e^3 \|n_e (\nu \nabla \e_u - e_\pi I) \|^2_{L^2(e)}\Big)^{\frac12}
		( \| \bz \|_{W^{2,2}} + \| s \|_{W^{1,2}}) \Big). 
	\end{aligned}
	\end{equation}
	The statement of the theorem follows by applying the elliptic regularity estimate \eqref{eq:ellregStokes} and dividing both sides by $\| \e_u\|_{L^2}$.
	
\printbibliography

\end{document}